\documentclass[12pt,twoside]{amsart}

\usepackage[english]{babel}
\usepackage{amsfonts,amssymb,amsthm,amsmath,mathtools,accents,latexsym}

\usepackage[a4paper,top=3cm,bottom=3cm,left=2.5cm,right=2.5cm,marginparwidth=1.75cm]{geometry}

\usepackage{xcolor}
\usepackage{graphicx,comment,mathtools}
\usepackage[colorlinks=true, allcolors=blue]{hyperref}
\usepackage{cleveref}
\newtheorem{theorem}{Theorem}[section] 
\newtheorem{lemma}[theorem]{Lemma}

\newtheorem{corollary}[theorem]{Corollary}

\theoremstyle{definition}
\newtheorem{definition}[theorem]{Definition}
\newtheorem{remark}[theorem]{Remark}
\newtheorem{example}[theorem]{Example}

\usepackage{tikz}

\usepackage{hyperref}
\hypersetup{
colorlinks=true,
linkcolor=blue,
filecolor=magenta,
urlcolor=cyan,
citecolor=red
}
\DeclareMathOperator{\tor}{Tor}

\DeclareMathOperator{\pd}{pd}
\DeclareMathOperator{\reg}{reg}

\newcommand{\B}{\mathcal{B}}

\title{Partial Betti splittings with applications to binomial edge ideals}
\date{\today   }

\author[A.V. Jayanthan]{A.V. Jayanthan}
\address[A.V. Jayanthan]
{Department of Mathematics, Indian Institute of Technology Madras, Chennai, Tamil Nadu, India - 600036}
\email{jayanav@iitm.ac.in }

\author[A. Sivakumar]{Aniketh Sivakumar}
\address[A. Sivakumar]
{Department of Mathematics, Tulane University, New Oreans, LA, 70118}
\email{asivakumar@tulane.edu}

\author[A. Van Tuyl]{Adam Van Tuyl}
\address[A. Van Tuyl]
{Department of Mathematics and Statistics\\
McMaster University, Hamilton, ON, L8S 4L8}
\email{vantuyla@mcmaster.ca}


\begin{document}


\begin{abstract}
We introduce the notion of a partial Betti splitting
of a homogeneous ideal, generalizing the notion
of a Betti splitting first given by Francisco, H\`a, and 
Van Tuyl.  Given a homogeneous ideal $I$ and two
ideals $J$ and $K$ such that $I = J+K$, a partial Betti splitting of $I$
relates {\it some} of the graded Betti of $I$ with those
of $J, K$, and $J\cap K$.   
As an application, we focus on the partial Betti
splittings of binomial edge ideals.  Using this new technique,
we generalize results of Saeedi Madani and Kiani related
to binomial edge ideals with cut edges, we describe a partial
Betti splitting for all binomial edge ideals, 
and we compute the total second
Betti number of binomial edge ideals of trees.
\end{abstract}

\maketitle

\section{Introduction}

Given a homogeneous ideal $I$ of a polynomial ring 
$R = k[x_1,\ldots,x_n]$ over an arbitrary field $k$, one is often
interested in the numbers $\beta_{i,j}(I)$, the graded Betti numbers of $I$, 
that are encoded into
the graded minimal free resolution of $I$.  In some situations, we 
can compute these numbers
by ``splitting'' the ideal $I$ into smaller ideals and use the graded
Betti numbers of these new ideals to find those of the ideal $I$.
More formally, suppose $\mathfrak{G}(L)$ denotes a set 
of minimal generators of a homogeneous ideal $L$.  Given a
homogeneous ideal $I$, we can ``split'' this ideal
as $I = J+K$ where $\mathfrak{G}(I)$ is the disjoint union of
$\mathfrak{G}(J)$ and $\mathfrak{G}(K)$.   The ideals $I, J, K$ and
$J \cap K$ are then related by the short exact sequence
$$0 \longrightarrow J\cap K \longrightarrow J \oplus K \longrightarrow J+K = I
\longrightarrow 0.$$
The mapping cone construction then implies that the graded Betti
numbers of $I$ satisfy
\begin{equation}\label{bettisplit}
\beta_{i,j}(I) \leq \beta_{i,j}(J) + \beta_{i,j}(K) + 
\beta_{i-1,j}(J \cap K) ~~\mbox{for all $i,j \geq 0$}.
\end{equation}
Francisco, H\`a, and Van Tuyl \cite{francisco_splittings_2008} defined
$I = J+K$ to be a {\it Betti splitting} if the above inequality
is an equality for all $i,j \geq 0$.

Betti splittings of monomial ideals first appeared in work 
of Eliahou and Kervaire \cite{EK1990}, 
Fatabbi \cite{fatabbi2001}, and Valla \cite{Valla2005}.  In fact, these prototypical
results  provided the inspiration for Francisco, 
H\`a, and Van Tuyl's introduction of Betti splittings in \cite{francisco_splittings_2008}. Their
paper also provided conditions on when
one can find Betti splittings of edge ideals, a monomial ideal associated to a graph (see \cite{francisco_splittings_2008} for more details). 
Betti splittings have proven to be a useful tool,
having been used to study: the graded Betti 
numbers of weighted edge ideals \cite{kara2022},  
the classification of Stanley-Reisner 
ideals of vertex decomposable ideals \cite{moradi2016},
the linearity defect of an ideal \cite{hop2016},
the depth function \cite{ficarra2023},
componentwise linearity \cite{bolognini2016},
and the Betti numbers of toric ideals \cite{FAVACCHIO2021409,gimenez2024}.

In general, an ideal $I$ may not have any Betti splitting.
However, it is possible that 
\Cref{bettisplit} may hold for {\it some} $i,j \geq 0$.
In order to quantify this behaviour, 
we introduce a new concept called
a {\it partial Betti splitting} of an ideal $I$.  Specifically,
if $I = J+K$ with $\mathfrak{G}(I)$ equal to the disjoint union
$\mathfrak{G}(J) \cup \mathfrak{G}(K)$, then 
$I = J+K$ is an {\it $(r,s)$-Betti splitting} if
    \[\beta_{i,j}(I) = \beta_{i,j}(J)+\beta_{i,j}(K)+\beta_{i-1, j}(J\cap K )\text{\hspace{3mm} for all $(i,j)$ with $i\geq r$ or $j\geq i+s$}.\]
Using the language of Betti tables, 
if $I = J+K$ is an $(r,s)$-Betti splitting, then 
all the Betti numbers in the $r$-th column and beyond 
or the $s$-th row  and beyond of the Betti table of $I$  satisfy \Cref{bettisplit}.  The Betti splittings of 
\cite{francisco_splittings_2008} will now called {\it
complete Betti splittings}.

The goal of this paper is two-fold.   First, we 
wish to develop the properties of partial Betti
splittings, extending the results of \cite{francisco_splittings_2008}.  Note that \cite{francisco_splittings_2008} focused on
Betti splittings of monomial ideals; however, as we show,
almost all the same arguments work for any 
homogeneous ideal $I$ of $R = k[x_1,\ldots,x_n]$ when 
$R$ is graded by a monoid $M$.   Among our results,
we develop necessary conditions for an $(r,s)$-Betti splitting:

\begin{theorem}[\Cref{parcon2}]
 Let $I$, $J$ and $K$ be homogeneous ideals of $R$
    with respect to the standard $\mathbb{N}$-grading such that $\mathfrak{G}(I)$ is the disjoint union of $\mathfrak{G}(J)$ and $\mathfrak{G}(K)$.  
    Suppose that there are integers $r$ and $s$ such that
    for all $i \geq r$ or $j \geq i+s$, $\beta_{i-1,j}(J \cap K) > 0$
    implies that $\beta_{i-1,j}(J) = 0$ and $\beta_{i-1,j}(K) = 0$.
    Then $I = J + K$ is an $(r,s)$-Betti splitting.
\end{theorem}

Second, we wish to illustrate (partial) 
Betti splittings by considering splittings of binomial
edge ideals.  If $G = (V(G,E(G))$ is a
graph on the vertex set $V = [n] :=\{1,\ldots,n\}$ 
and edge set
$E$, the {\it binomial edge ideal of $G$} is the binomial
ideal 
$J_G = \langle x_iy_j - x_jy_i ~|~ \{i,j\} \in E \rangle$ 
in the polynomial ring 
$R = k[x_1,\ldots,x_n,y_1,\ldots,y_n]$. 
Binomial edge ideals, 
which were first introduced 
in \cite{herzog_binomial_2010,Ohtani2011}, 
have connections to algebraic statistics, among other areas.  The past decade has 
seen a flurry of
new results about the homological invariants (e.g., Betti numbers,
regularity, projective dimension) for this family of ideals (see \cite{ZZ13}, \cite{SZ14}, \cite{deAlba_Hoang_18}, \cite{herzog_extremal_2018}, \cite{KS20}, \cite{jayanthan_almost_2021} for a partial list on the Betti numbers of binomial edge ideals). 
Interestingly, Betti splittings of binomial
edge ideals have not received any attention, 
providing additional motivation to study this family of ideals. 

In order to split $J_G$, we wish to partition the generating set $\mathfrak{G}(J_G)$ in such a 
way that the resulting ideals 
generated by each partition, say $J$ and $K$, are the binomial edge 
ideals of some subgraphs of $G$, that is, 
splittings of the form $J_G = J_{G_1}+J_{G_2}$ where $G_1$ and $G_2$ are 
subgraphs.   We focus on two natural 
candidates. The first  way
is to fix an edge $e = \{i,j\} \in E(G)$ and consider the splitting
$$J_G = J_{G\setminus e} + \langle x_iy_j- x_jy_i  \rangle.$$
where $G\setminus e$ denotes the graph $G$ with the
edge $e$ removed. 
The second way is to fix a vertex $s \in V(G)$ and consider
the set $F \subseteq E(G)$ of all edges that contain the vertex $s$.  We can
then split $J_G$ as follows
$$J_G = \langle x_sy_j-x_jy_s ~|~ \{s,j\} \in F \rangle + 
\langle x_ky_j-x_jy_k ~|~ \{k,l\} \in E(G) \setminus F \rangle.$$
We call such a partition an $s$-partition of $G$.
Note that the first ideal is the binomial edge ideal
of a star graph, while the second ideal is the binomial
edge ideal of the graph $G \setminus \{s\}$, 
the graph with the vertex $s$ removed.
These splittings are reminiscent of 
the edge splitting
of edge ideals and the $x_i$-splittings of monomial
ideals introduced in \cite{francisco_splittings_2008}.

In general, neither of these splitting will give us a 
complete Betti splitting.  This is not too surprising 
since the edge ideal analogues are not always complete 
Betti splittings.  So it is natural
to ask when we have a partial or complete  Betti
splitting using
either division of $J_G$.  
Among our results in Section 4, we give
a sufficient condition on an edge $e$ of
$G$ so that the first partition gives a complete
Betti splitting.  In the statement below,
an edge is a cut-edge if $G \setminus e$ has
more connected components than $G$, and a vertex
is free if it belongs to a unique maximal clique, a subset
of vertices
of $G$ such that all the vertices are all adjacent to each other.  

\begin{theorem}[\Cref{singlefreevertex}]\label{them2}
    Let $e = \{u,v\} \in E(G)$ be a cut-edge where $v$ is a free vertex in $G\setminus e$. Then
    $J_G = J_{G\setminus e}+\langle x_uy_v-x_vy_u\rangle$ is a complete Betti splitting.
\end{theorem}

\noindent
Theorem \ref{them2} generalizes 
previous work of
Saeedi Madani and Kiani \cite{kiani_regularity_2013-1},
and it allows us to give new proofs for their
results about the Betti numbers, regularity, and
projective dimension for some classes
of binomial edge ideals (see \Cref{freecutedge}).

In the case of $s$-partitions, we again do not always have
a complete Betti splitting.  However, we can derive
a result about the partial Betti splittings
for all graphs. 

\begin{theorem}[\Cref{maintheo2}]
    Let $J_G$ be the binomial edge ideal of a graph $G$ and let $J_G = J_{G_1}+J_{G_2}$ be an $s$-partition of $G$. Let $c(s)$ be the size of the largest clique that contains $s$. Then
    $$
        \beta_{i,j}(J_G) = \beta_{i,j}(J_{G_1})+\beta_{i,j}(J_{G_2})+\beta_{i-1, j}(J_{G_1}\cap J_{G_2})~~~
        \mbox{for all $(i,j)$ with $i\geq c(s)$ or $j\geq i+4$.}
    $$
    In other words, $J_G = J_{G_1}+J_{G_2}$ is a 
    $(c(s), 4)$-Betti splitting.
\end{theorem}

\noindent
Note that if $G$ is a triangle-free graph, then for 
every vertex
$i \in V(G)$ we have $c(i) \leq 2$.  We can use the above result
to construct a complete Betti splitting for the
binomial edge ideals of all triangle-free graphs (see Corollary \ref{trianglefree}).

In the final section, we use the complete Betti splitting
of \Cref{them2} to explore the (total) graded Betti numbers of binomial 
edge ideals of trees.   In particular, we give
formulas for the first and second total Betti numbers for the binomial edge ideal of any tree. Our result  extends work
of Jayanthan, Kumar, and Sarkar \cite{jayanthan_almost_2021} which computed
the first total Betti numbers for these ideals.

Our paper is structured as follows.  In Section 2 we recall the 
relevant background.  In Section 3 we introduce
the notion of a partial Betti splitting and describe 
some of their basic properties.  In Section 4, we consider
splittings of $J_G$ using a single edge of $G$, while in
Section 5, we consider a splitting of $J_G$ by partitioning
the generators on whether or not they contain $x_s$ or $y_s$ for 
a fixed vertex $s$.
In our final section we determine the second
total Betti number of binomial edge ideals of trees.


\section{Preliminaries}

In this section we recall the relevant background on Betti numbers,
graph theory, and binomial edge ideals that
is required for later results.  

\subsection{Homological algebra}
Throughout
this paper $k$ will denote an arbitrary field.  Let 
$R = k[x_1,\ldots,x_n]$ be a polynomial ring over $k$.  
We will use various gradings of $R$. Recall
that if $M$ is a monoid (a set with an addition operation and
additive identity), we say a ring $S$ is {\it $M$-graded} if
we can write $S = \bigoplus_{j \in M} S_j$, where
each $S_j$ is an additive group and $S_{j_1}S_{j_2} \subseteq 
S_{j_1+j_2}$ for all $j_1,j_2 \in M$.  We will primarily
use three gradings of $R$ in this paper: (1) $R$ has an $\mathbb{N}$-grading
by setting $\deg(x_i) = 1$ for all $i$; (2) $R$ has an $\mathbb{N}^n$-grading
by setting $\deg(x_i) = e_i$ for all $i$, where
$e_i$ is the standard basis element of $\mathbb{N}^n$;
and (3) $R$ has an $\mathbb{N}^2$-grading by
setting the degree of some of the $x_i$'s to $(1,0)$,
and the degrees of the rest of the $x_i$'s to $(0,1)$.

Given an $M$-graded ring $R$, an element $f \in R$ is 
{\it homogeneous} if $f \in R_j$ for some $j \in M$.
We say the {\it degree} of $f$ is $j$ and write
$\deg(f) = j$.
An ideal $I \subseteq R$ is
{\it homogeneous} if it is generated by
homogeneous elements.  We write $I_j$ to denote
all the homogeneous elements of degree $j\in M$ in $I$. We let $\mathfrak{G}(I)$ denote
a minimal set of homogeneous generators of $I$.  
While the choice of elements of
$\mathfrak{G}(I)$ may not be unique, the number of
generators of a particular degree is an
invariant of the ideal.   If $I$ is
a homogeneous ideal, then the 
Tor modules ${\rm Tor}_i(k,I)$ are also $M$-graded
for all $i \geq 0$.
The {\it $(i,j)$-th graded Betti number of
$I$} is then defined to be
$$\beta_{i,j}(I) := \dim_k {\rm Tor}_i(k,I)_j 
~~\mbox{for $i \in \mathbb{N}$ and $j \in M$.}$$
We use the convention that $\beta_{i,j}(I) = 0$ if $i <0$.
We are sometimes interested in the (multi)-graded Betti
numbers of the quotient $R/I$;  we make use of the
identity $\beta_{i,j}(R/I) = \beta_{i-1,j}(I)$ for all $i \geq 1$
and $j \in M$.  The graded Betti number $\beta_{i,j}(I)$ is also equal
to the number of syzygies of degree $j$ in the $i$-th
syzygy module of $I$.  For further details, see the book
of Peeva \cite{P2011}.

When $R$ has the standard $\mathbb{N}$-grading, we are also
interested in the following two invariants: the {\it (Castelnuovo-Mumford) regularity of $I$}, which is defined as 
$${\rm reg}(I) = \max\{ j-i ~|~ \beta_{i,i+j}(I) \neq 0\},$$
and the {\it projective dimension of $I$}, which is defined as
$${\rm pd}(I) = \max\{i ~|~ \beta_{i,j}(I) \neq 0\}.$$
These invariants measure the ``size'' of the minimal graded free
resolution of $I$.

\subsection{Graph theory}
Throughout this paper, we use $G = (V(G),E(G))$ to 
represent a finite simple graph where $V(G)$
denotes the vertices and $E(G)$ denotes the edges.  
Most of our graphs will have the vertex set $[n] = \{1,\dots ,n\}$. 

A {\it subgraph} of $G$ is a graph $H$ such 
that $V(H)\subseteq V(G)$ and $E(H)\subseteq E(G)$. An 
\textit{induced subgraph} on $S\subset V(G)$, denoted by $G[S]$, is a the subgraph with vertex set $S$ and 
for all $u,v\in S$, if $\{u,v\}\in E(G)$, then $
\{u,v\}\in E(G[S])$. The {\it complement} of a graph, 
denoted $G^c$, is a graph with 
$V(G^c) = V(G)$ and $E(G^c) = \{\{u,v\}\mid \{u,v\}\notin E(G)\}$. 

From a given graph $G = (V(G),E(G))$, if $e \in E(G)$, then
we denote by $G\setminus e$ the subgraph of $G$ on the same
vertex set, but edge set $E(G\setminus e) = E(G) \setminus \{e\}$.  
Given any $i \in V(G)$, we let $N_G(i) = \{j ~|~ \{i,j\} \in E(G)\}$
denote the set of {\it neighbours} of the vertex $i$.  
The {\it degree} of a vertex $i$ is then $\deg_G i  = |N_G(i)|$. In the context where there is a fixed underlying graph, we omit the subscript $G$ and write this as $\deg i$.
The {\it closed neighbourhood of $i$} is the set
$N_G[i] =N_G(i) \cup \{i\}$.
If $G = (V(G),E(G))$ is a graph and $e =\{i,j\} \not\in E(G)$, 
we let $G_e$ denote the graph on $V(G)$, but with edge
set 
$$E(G_e) = E(G) \cup \{\{k,l\} ~|~ k,l \in N_G(i)~~\mbox{or}~~k,l \in N_G(j) \}.$$
So, $G$ is a subgraph $G_e$.

We  will require a number of special families of graphs.
The \textit{$n$-cycle}, denoted $C_n$, is the graph
with vertex set  $[n]$ with $n \geq 3$ and edge 
set $\{\{i,i+1\} ~|~ i =1,\ldots,n-1\} \cup \{\{1,n\}\}.$
A \textit{chordal graph} $G$ is a graph where 
all the induced subgraphs of $G$
that are cycles are 3-cycles, that is, there are no 
induced $n$-cycles with $n\geq 4$. 
A \textit{triangle-free graph} is a graph $G$ 
such that $C_3$ is not an induced subgraph of $G$.
A \textit{tree} is a graph which has no induced cycles.  A particular example of a tree that we will use is the {\it star graph} on $n$ vertices,
denoted $S_n$.  Specifically, $S_n$ is the graph on the vertex
set $[n]$ and edge set $E(S_n) = \{\{1,k\}\mid 1<k\leq n\}$.

A \textit{complete graph} is a graph $G$ where  $\{u,v\}\in E(G)$ for 
all $u,v\in V(G)$.   If $G$ is a complete graph on $[n]$, we denote
it by $K_n$. A \textit{clique} in a graph $G$
is an induced subgraph $G[S]$ that is a complete
graph. A \textit{maximal clique} is a clique that is not contained 
in any larger clique.

A vertex $v$
of $G$ is a \textit{free vertex} if $v$ only belongs to
a unique maximal clique in $G$, or equivalently, the induced
graph on $N_G(v)$ is a clique.
An edge $e = \{u,v\}$ in $G$ is a \textit{cut edge} if its deletion from $G$ yields a graph with more connected components than $G$.  
Note that a tree is a graph where all of its edges are cut edges.
A \textit{free cut edge} is a cut edge $\{u,v\}$ such that both ends, $u$ and $v$, are free vertices in $G \setminus e$.

We are also interested in cliques combined with other graphs.
A graph $G$ is said to be a \textit{clique-sum} of $G_1$ and $G_2$, 
denoted by $G = G_1 \cup_{K_r} G_2$, 
if $V(G_1) \cup V(G_2) = V(G)$, $E(G_1) \cup E(G_2) = E(G)$ and
the induced graph on
$V(G_1) \cap V(G_2)$ is the clique $K_r$. If $r = 1$, then we write $G = G_1 \cup_v G_2$ for the clique-sum $G_1 \cup _{K_1} G_s$ where $V(K_1) = \{v\}$.  A graph $G$ is \textit{decomposable} 
if there exists subgraphs $G_1$ and $G_2$ such that $G_1\cup_{v}G_2 = G$ and $v$ is a free vertex of $G_1$ and $G_2$.  
So a decomposable graph is an example
of a clique-sum  on a $K_1$ where the $K_1$ is a free vertex in both
subgraphs.

\begin{example}
    Consider the graph $G$ in \Cref{fig:graph5}, with $V(G) = [7]$ and 
    $$E(G) = \{\{1,2\}, \{2,3\}, \\\{2,4\}, \{4,5\}, \{4,6\}, \{4,7\}, \{6,7\}\}.$$ Here, we can see that $G = T \cup_{\{4\}} K_3$, where $T$ is the tree with $V(T) = \{1,2,3,4,5\}$ and $E(T) = \{\{1,2\}, \{2,3\}, \{2,4\}, \{4,5\}\}$ and $K_3$ is the clique of size $3$, with $V(K_3) = \{4,6,7\}$ and $E(K_3) = \{\{4,6\}, \{4,7\}, \{6,7\}\}$.

    \begin{figure}[ht]
        \centering
\begin{tikzpicture}[every node/.style={circle, draw, fill=white!60, inner sep=2pt}, node distance=1.5cm]
    \node (1) at (0, 0) {1};
    \node (2) at (1.5, 0) {2};
    \node (3) at (3, 0) {3};
    \node (4) at (1.5, -1.5) {4};
    \node (5) at (0, -1.5) {5}; 
    \node (6) at (0.5, -2.5) {6};
    \node (7) at (2.5, -2.5) {7};

    \draw (1) -- (2);
    \draw (2) -- (3);
    \draw (2) -- (4);
    \draw (4) -- (5);
    \draw (4) -- (6);
    \draw (4) -- (7);
    \draw (6) -- (7);
\end{tikzpicture}

        \caption{$G = T\cup_{\{4\}}K_3$}
        \label{fig:graph5}
    \end{figure}
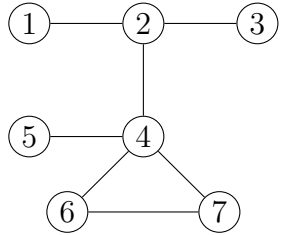
\end{example}


\subsection{Binomial edge ideals}
Suppose that $G = (V(G),E(G))$ is a finite simple graph
with $V(G) = [n]$.  The {\it binomial edge ideal} of $G$,
denoted $J_G$, is the binomial ideal
$$J_G = \langle x_iy_j - x_jy_i ~|~ \{i,j\} \in E(G) \rangle$$
in the polynomial ring $R = k[x_1,\ldots,x_n,y_1,\ldots,y_n]$. 
In what follows, we will find it convenient to consider
different gradings of $R$;  we can 
grade the polynomial ring $R$ either with the standard grading 
where $\deg x_i=\deg y_i=1$ for all $i$, with 
an $\mathbb{N}^n$-multigrading 
where $\deg x_i=\deg y_i=(0,\dots,1,\dots, 0)$, the $i$-th unit vector
for all $i$,
or with an $\mathbb{N}^2$-grading where $\deg x_i = (1,0)$ for 
all $i$ and $\deg y_j = (0,1)$ for all $j$. Note that
$J_G$ is a homogeneous ideal with respect to all three gradings.

We review some useful facts from the literature about the idea
$J_G$.  Recall that a standard graded ideal
$I$ has {\it linear resolution} if $I$ is generated by
homogeneous elements of degree $d$  and $\beta_{i,i+j}(I) = 0$
for all $j \neq d$. 

\begin{theorem}\label{completebetti}
Let $G = K_n$ be a complete graph. Then
\begin{enumerate}   
\item
    The binomial edge ideal  $J_G$ has a linear resolution.
    \item $\beta_{i,i+2}(J_G) = (i+1)\binom{n}{i+2}$ for $i \geq 0$ and $0$ otherwise.
    \end{enumerate}
\end{theorem}

\begin{proof}
Statement (1) follows from  {\cite[Theorem 2.1]{kiani_binomial_2012}}. Statement (2)
follows from a more general fact of
Herzog, Kiani, and Saaedi Madani \cite[Corollary 4.3]{herzog_linear_2017} on the Betti numbers
that appear in the linear 
strand of a binomial edge ideals applied to $K_n$.
\end{proof}
The next result is  related to a cut edge in a graph.

\begin{lemma}[{\cite[Theorem 3.4]{mohammadi_hilbert_2014}}]\label{lemma 3.8}
    Let $G$ be a simple graph and let $e = \{i,j\}\notin E(G)$ be a cut
    edge in $G\cup \{e\}$. Let $f_e = x_iy_j-x_jy_i$. Then
    $J_G:\langle f_e \rangle = J_{G_e}$.
\end{lemma}

We will require the next result about the
Betti polynomials of binomial edge ideals of  decomposable graphs. For an 
$\mathbb{N}$-graded $R$-module $M$, 
the {\it Betti polynomial of $M$} is 
$$B_M(s,t) = \sum_{i,j \geq 0} \beta_{i,j}(M)s^it^j.$$
The following result is due to Herzog and Rinaldo, which
generalized an earlier result of of Rinaldo and Rauf \cite{rauf_construction_2014}.
\begin{theorem}[{\cite[Proposition 3]{herzog_extremal_2018}}]\label{freevertexbetti}
    Suppose that  $G$ is a decomposable graph
    with decomposition $G = G_1\cup G_2$. Then
    \[B_{R/J_G}(s, t) = B_{R/J_{G_1}}(s, t)B_{R/J_{G_2}}(s, t).\]
\end{theorem}

The graded Betti numbers in the linear strand of $J_G$ (all
the Betti numbers of the form $\beta_{i,i+2}(J_G))$ were 
first calculated by Herzog, Kaini, and Saeedi Madani.
In the statement below, 
$\Delta(G)$ is the clique complex of the graph $G$ and $f_{i+1}(\Delta(G))$ is the number of faces in $\Delta(G)$ of dimension $i+1$.

\begin{theorem}[{\cite[Corollary 4.3]{herzog_linear_2017}}]\label{linearbinom}
    Let $G$ be a finite simple graph with binomial edge
    ideal $J_G$. Then the Betti numbers in the linear
    strand of $J_G$ are given by
   \[\beta_{i,i+2}(J_G) = (i+1)f_{i+1}(\Delta(G)) ~~\mbox{for $i\geq 0$.}\]
   \end{theorem}

\begin{example}\label{runningexample}
Let $G$ be the finite simple graph on the vertex set $[7]$ with edge set
$$E(G) =\{\{1,2\}, \{1,3\}, \{1,4\}, \{1, 5\}, \{1,7\},\{2, 4\}), \{2,5\}, 
\{2,7\},\{3,7\},\{4,5\},\{6,7\}\}.$$
This graph is drawn in Figure \ref{fig:runningexamp}.     
    \begin{figure}[ht]
        \centering
\begin{tikzpicture}[every node/.style={circle, draw, fill=white!60, inner sep=2pt}, node distance=1.5cm]
    \node (1) at (1.5, 0) {1};
    \node (2) at (1.5, -1.5) {2};
    \node (3) at (3, 0) {3};
    \node (4) at (0, -1.5) {4};
    \node (5) at (0, 0) {5}; 
    \node (6) at (4.5, 0) {6};
    \node (7) at (3, -1.5) {7};

    \draw (1) -- (2);
    \draw (1) -- (3);
    \draw (1) -- (4);
    \draw (1) -- (5);
    \draw (1) -- (7);
    \draw (2) -- (4);
    \draw (2) -- (5);
    \draw (2) -- (7);
    \draw (3) -- (7);
    \draw (4) -- (5);
    \draw (6) -- (7);
\end{tikzpicture}
        \caption{Graph $G$}
        \label{fig:runningexamp}
    \end{figure}
    The binomial edge ideal of $G$ is an ideal of 
    $R=k[x_1,\ldots,x_7,y_1,\ldots,y_7]$ with 11 generators.  
    Specifically,
    \begin{multline*}
    J_G = \langle x_1y_2-x_2y_1, 
    x_1y_3-x_3y_1,
    x_1y_4-x_4y_1,
    x_1y_5-x_5y_1,
    x_1y_7-x_7y_1,
    x_2y_4-x_4y_2, \\
    x_2y_5-x_5y_2, 
    x_2y_7-x_7y_2, 
    x_3y_7-x_7y_3,
    x_4y_5-x_5y_4,
    x_6y_7-x_7x_6 \rangle.
    \end{multline*}
\end{example}


\section{Partial Betti splittings}

In this section, we define the notion of a partial Betti splitting, generalising the concept of a Betti splitting first established 
by Francisco, H\`a, and Van Tuyl \cite{francisco_splittings_2008}.  While a Betti
splitting of an ideal $I$ is a ``splitting" of $I$ into
two ideals $I = J+K$
such that {\it all} of the (multi)-graded Betti
numbers of $I$ can be related to those of $J, K$ and
$J \cap K$, in a partial Betti splitting, we only 
require that some of these relations to hold. 
Betti splittings of ideals were 
originally defined just for
monomial ideals, since the original motivation
of \cite{francisco_splittings_2008} was to extend 
Eliahou and Kevaire's  splitting of monomial ideals 
\cite{EK1990}.  However,
a careful examination of the proofs of \cite{francisco_splittings_2008} reveals that 
some of the main results hold for all (multi)-graded
ideals in a polynomial ring $R = k[x_1,\ldots,x_n]$.  We develop
partial Betti splittings within this more general context.

Assuming that $R$ is $M$-graded, let $I,J$, and $K$ be
homogeneous ideals with respect to this grading 
such that $I = J + K$ and
$\mathfrak{G}(I)$ is the disjoint union of $\mathfrak{G}(J)$
and $\mathfrak{G}(K)$.  We have a natural short exact
sequence
$$0 \longrightarrow J \cap K \stackrel{\varphi}{\longrightarrow} J \oplus K 
\stackrel{\psi}{\longrightarrow} I = J+K \longrightarrow 0,$$
where the maps $\varphi(f) = (f,-f)$ and $\psi(g,h) = g+h$ have degree $0$, i.e.,
they map elements of degree $j \in M$ to elements of degree
$j \in M$.
The mapping cone resolution applied to this
short exact sequence then implies that
$$\beta_{i,j}(I) \leq \beta_{i,j}(J) + \beta_{i,j}(K)
+ \beta_{i-1,j}(J \cap K) ~~\mbox{for all $i \geq 0$ and $j \in M$}.$$
We are then interested in when we have an equality.
The following lemma gives such a condition for
a specific $i \in \mathbb{N}$ and $j \in M$. The 
proof is essentially the same as \cite[Proposition 2.1]{francisco_splittings_2008} which considered only monomial ideals, 
but for completeness, we have included the details here.

\begin{lemma}\label{singlesplit}
    Let $R$ be a $M$-graded ring, and suppose that
    $I, J$, and $K$ are homogeneous ideals with respect to this
    grading such that $I = J+K$ and $\mathfrak{G}(I)$ is the disjoint
    union of $\mathfrak{G}(J)$ and $\mathfrak{G}(K)$.  
    Let $$0 \longrightarrow J \cap K \stackrel{\varphi}{\longrightarrow} J \oplus K 
\stackrel{\psi}{\longrightarrow} I = J+K \longrightarrow 0$$
be the natural short exact sequence.  Then, 
for a fixed integer $i > 0$ and $j \in M$, the following
two statements are equivalent:
\begin{enumerate}
    \item $\beta_{i,j}(I) = \beta_{i,j}(J)+\beta_{i,j}(K) +
    \beta_{i-1,j}(J\cap K)$;
    \item the two maps $$\varphi_i:{\rm Tor}_i(k,J \cap K)_j 
    \rightarrow {\rm Tor}_i(k,J)_j \oplus {\rm Tor}_i(k,K)_j$$ 
    and
    $$\varphi_{i-1}:{\rm Tor}_{i-1}(k,J \cap K)_j 
    \rightarrow {\rm Tor}_{i-1}(k,J)_j \oplus {\rm Tor}_{i-1}(k,K)_j$$
    induced from the long exact sequence of \emph{Tor} using the above
    short sequence are both the zero map. 
\end{enumerate}
\end{lemma}

\begin{proof}  Fix an integer $i >0$ and $j \in M$.
    Using the short exact sequence given in the statement, we can
    use Tor to create a long exact sequence that satisfies
    \begin{multline*}
     \cdots \rightarrow 
    {\rm Tor}_i(k,J \cap K)_j 
    \stackrel{\varphi_i}{\rightarrow} 
    {\rm Tor}_i(k,J)_j \oplus {\rm Tor}_i(k,K)_j
    \rightarrow {\rm Tor}_i(k,I)_j \rightarrow \\
    {\rm Tor}_{i-1}(k,J \cap K)_j 
    \stackrel{\varphi_{i-1}}\rightarrow {\rm Tor}_{i-1}(k,J)_j \oplus {\rm Tor}_{i-1}(k,K)_j
    \rightarrow \cdots . 
    \end{multline*}
    Consequently, we have an exact sequence of vector spaces
     \begin{multline*}
    0 \rightarrow 
    {\rm Im}(\varphi_i)_j \rightarrow 
    {\rm Tor}_i(k,J)_j \oplus {\rm Tor}_i(k,K)_j
    \rightarrow {\rm Tor}_i(k,I)_j \rightarrow \\
    {\rm Tor}_{i-1}(k,J \cap K)_j 
    \stackrel{\varphi_{i-1}}\rightarrow A_j
    \rightarrow 0
    \end{multline*}
    where $$A  = {\rm Im}(\varphi_{i-1}) \cong {\rm Tor}(k,J \cap K)/{\ker \varphi_{i-1}}.$$
    We thus have 
    $$\beta_{i,j}(I) = \beta_{i,j}(J)+\beta_{i,j}(K) + \beta_{i-1,j}(J\cap K)
    - \dim_k ({\rm Im}(\varphi_i))_j - \dim_k ({\rm Im}(\varphi_{i-1}))_j.$$

    To prove $(1) \Rightarrow (2)$, note that
    if both $\varphi_i$ and $\varphi_{i-1}$ are the zero
    map, we have $\beta_{i,j}(I) = \beta_{i,j}(J) + \beta_{i,j}(K) + 
    \beta_{i-1,j}(J \cap K)$.  For $(2) \Rightarrow
    (1)$, if either of $\varphi_i$ or $\varphi_{i-1}$
    is not the zero map, either $\dim_k ({\rm Im}(\varphi_i))_j > 0$ or
    $\dim_k ({\rm Im}(\varphi_{i-1}))_j> 0$, which forces
    $\beta_{i,j}(I) <  \beta_{i,j}(J) + \beta_{i,j}(K) + 
    \beta_{i-1,j}(J \cap K).$
\end{proof}

The following corollary, which is \cite[Proposition 3]{francisco_splittings_2008},
immediately follows.

\begin{corollary}
    Let $R$ be a $M$-graded ring, and suppose that
    $I, J$, and $K$ are homogeneous ideals with respect to this
    grading such that $I = J+K$ and $\mathfrak{G}(I)$ is the disjoint
    union of $\mathfrak{G}(J)$ and $\mathfrak{G}(K)$.  
    Let $$0 \longrightarrow J \cap K \stackrel{\varphi}{\longrightarrow} J \oplus K 
\stackrel{\psi}{\longrightarrow} I = J+K \longrightarrow 0$$
be the natural short exact sequence.  Then
$\beta_{i,j}(I) = \beta_{i,j}(J)+\beta_{i,j}(K) +
    \beta_{i-1,j}(J\cap K)$ for all integers $i \geq 0$ and $j \in M$,
    if and only if 
    the maps $$\varphi_i:{\rm Tor}_i(k,J \cap K)_j 
    \rightarrow {\rm Tor}_i(k,J)_j \oplus {\rm Tor}_i(k,K)_j$$ 
    induced from the long exact sequence of {\rm Tor} using the 
    above short exact sequence 
    are the zero map for all integers $i \geq 0$ and $j \in M$.
\end{corollary}

Applying  \Cref{singlesplit} directly implies that we would
need to understand the induced maps between {\rm Tor} modules
in order to determine if a specific $(i,j)$-th graded
Betti number of $I$ can be determined from those of $J$, $K$, and $J\cap K$.
However, we can now modify Theorem 2.3 from \cite{francisco_splittings_2008} to obtain a 
a specific ``splitting'' of $\beta_{i,j}(I)$ from
other graded Betti numbers.

\begin{theorem}\label{parcon}
  Let $R$ be a $M$-graded ring, and suppose that
    $I, J$, and $K$ are homogeneous ideals with respect to this
    grading such that $I = J+K$ and $\mathfrak{G}(I)$ is the disjoint
    union of $\mathfrak{G}(J)$ and $\mathfrak{G}(K)$.  
    Suppose for a fixed integer $i > 0$ and $j \in M$ we have that:
    \begin{itemize}
        \item  if $\beta_{i,j}(J\cap K)>0$, then $\beta_{i,j}(J) = 0$ and $\beta_{i,j}(K) = 0$, and
        \item if $\beta_{i-1,j}(J\cap K)>0$, then $\beta_{i-1,j}(J) = 0$ and $\beta_{i-1,j}(K) = 0.$
    \end{itemize} Then we have:
    \begin{equation}
        \beta_{i,j}(I) = \beta_{i,j}(J)+\beta_{i,j}(K)+\beta_{i-1, j}(J\cap K ).
    \end{equation}
\end{theorem}
\begin{proof}
     Since $I = J+K$, we have the short exact sequence
    \[0\longrightarrow J\cap K \xlongrightarrow{\varphi} J\oplus K \xlongrightarrow{\psi} J+K = I\longrightarrow 0.\]
    For all integers $\ell \geq 0$ and $j \in M$, we 
    get the induced maps 
    $$\varphi_\ell:{\rm Tor}_\ell(k,J \cap K)_j 
    \rightarrow {\rm Tor}_\ell(k,J)_j \oplus {\rm Tor}_\ell(k,K)_j$$ 
    from the long exact sequence of {\rm Tor} using the 
    short exact sequence.

    Let $i > 0$ and $j \in M$ be the fixed 
    $i$ and $j$  as in the statement.  There are four
    cases to consider: (1) $\beta_{i,j}(J \cap K)$
    and $\beta_{i-,j}(J \cap K)$ both non-zero,
    (2) $\beta_{i,j}(J\cap K) = 0$ and $\beta_{i-1,j}(J \cap K) > 0$,
    (3) $\beta_{i,j}(J\cap K) > 0$ and $\beta_{i-1,j}(J \cap K) = 0$,
    and (4) both $\beta_{i,j}(J\cap K)  = \beta_{i-1,j}(J \cap K) = 0$.
    
    In case (1), the maps
    $\varphi_i$ and $\varphi_{i-1}$ must be the zero map
    since $0 =\beta_{i,j}(J)$ and
    $0 = \beta_{i,j}(K)$ imply that 
    ${\rm Tor}_i(k,J)_j \oplus {\rm Tor}_i(k,K)_j = 0$,
    and similarly, $0 =\beta_{i-1,j}(J)$ and
    $0 = \beta_{i-1,j}(K)$ imply
    ${\rm Tor}_{i-i}(k,J)_j \oplus {\rm Tor}_{i-1}(k,K)_j = 0$.
    The conclusion now follows from \Cref{singlesplit}.

    For case (2), the map $\varphi_{i-1}$ is the zero map using
    the same argument as above.  On the other hand,
    $0 = \beta_{i,j}(J \cap K) = \dim_k {\rm Tor}(k, J\cap K)_j$ implies
    that $\varphi_i$ is the zero map.  We now apply
    \Cref{singlesplit}.

    Cases (3) and (4) are proved similarly, so we omit the details.
    \end{proof}

We now introduce the notion of a partial Betti splitting, that weakens
the conditions of a Betti splitting found in \cite{francisco_splittings_2008}.
Note that we assume that $R$ has the standard
$\mathbb{N}$-grading.

\begin{definition}\label{pardef}
    Let $I$, $J$ and $K$ be homogeneous ideals of $R$
    with respect to the standard $\mathbb{N}$-grading such that $\mathfrak{G}(I)$ is the disjoint union of $\mathfrak{G}(J)$ and $\mathfrak{G}(K)$. Then $I= J + K$ is an {\it $(r,s)$-Betti splitting} if
    \[\beta_{i,j}(I) = \beta_{i,j}(J)+\beta_{i,j}(K)+\beta_{i-1, j}(J\cap K )\text{\hspace{3mm} for all $(i,j)$ with $i\geq r$ or $j\geq i+s$}.\]
    If $(r,s) \neq (0,0)$ we call an
    $(r,s)$-Betti splitting $I=J+K$ a {\it partial
    Betti splitting}.
    Otherwise, we say that $I = J+K$ is a {\it complete Betti splitting}
    if it is a $(0,0)$-Betti splitting, that is,
    $$\beta_{i,j}(I) = \beta_{i,j}(J) + \beta_{i,,j}(K) + \beta_{i-1,j}(J\cap K)
    ~~\mbox{for all $i,j \geq 0$}.$$
\end{definition}

\begin{remark}
    A complete Betti splitting is what Francisco, H\`a, and Van Tuyl \cite{francisco_splittings_2008}
    called a Betti splitting.
\end{remark}

\begin{remark}
We can interpret the above definition with the Betti table of $I$. 
The {\it Betti table of $I$} is a table whose columns are indexed
by the integers $i\geq 0$, and in row $j$ and 
column $i$, we place $\beta_{i,i+j}(I)$.
If $I = J+K$ is an $(r,s)$-Betti splitting, then all the Betti numbers
in the Betti table of $I$ in
the $r$-th column and beyond or in the $s$-th row and beyond 
are ``split'',
that is, they satisfy
$\beta_{i,j}(I) = \beta_{i,j}(J)+\beta_{i,j}(K)+\beta_{i-1, j}(J\cap K ).$
\end{remark}

The following observation will be useful.

\begin{lemma}
    Suppose that $I=J+K$ is an $(r,s)$-Betti splitting of $I$.
    If $r = 0$ or $1$, then $I=J+K$ is a complete Betti splitting.
\end{lemma}

\begin{proof}
Since $I = J+K$ is an $(r,s)$-Betti splitting, we 
have $\mathfrak{G}(I) = \mathfrak{G}(J) \cup
\mathfrak{G}(K)$.  Consequently, we always have
$$\beta_{0,j}(I) = \beta_{0,j}(J) + \beta_{0,j}(K) +
\beta_{-1,j}(J\cap K) = \beta_{0,j}(J)+\beta_{0,j}(K) ~\mbox{for $i=0$ and all
$j \geq 0$.}$$
For any $(r,s)$-Betti splitting with $r =0$ or $1$, the definition implies
\[\beta_{i,j}(I) = \beta_{i,j}(J)+\beta_{i,j}(K)+\beta_{i-1, j}(J\cap K ) ~\mbox{for all $i > 0$ and all $j \geq 0$}.\]
So, for any $i,j \geq 0$, we have $\beta_{i,j}(I) = \beta_{i,j}(J) + 
\beta_{i,j}(K) + \beta_{i-1,j}(J \cap K)$, that is, we have
a complete Betti splitting.
\end{proof}

We can now use Theorem \ref{parcon} to get a condition on
$(r,s)$-Betti splittings.

\begin{theorem}\label{parcon2}
 Let $I$, $J$ and $K$ be homogeneous ideals of $R$
    with respect to the standard $\mathbb{N}$-grading such that $\mathfrak{G}(I)$ is the disjoint union of $\mathfrak{G}(J)$ and $\mathfrak{G}(K)$.  
    Suppose that there are integers $r$ and $s$ such that
    for all $i \geq r$ or $j \geq i+s$, $\beta_{i-1,j}(J \cap K) > 0$
    implies that $\beta_{i-1,j}(J) = 0$ and $\beta_{i-1,j}(K) = 0$.
    Then $I = J + K$ is an $(r,s)$-Betti splitting.
\end{theorem}

\begin{proof}
    Let $r$ and $s$ be as in the statement, and
    suppose that $(i,j)$ is fixed integer tuple that satisfies $i \geq r$ or $j \geq i+s$.  But then
    $(i+1,j)$ also satisfies $i+1 \geq r$ or 
    $j \geq i+s$.  Consequently, for this fixed
    $(i,j)$, the hypotheses imply
    \begin{enumerate}
        \item[$\bullet$]
    if $\beta_{i-1,j}(J\cap K) >0$, then $\beta_{i-1,j}(J) = \beta_{i-1,j}(K) = 0$, and
    \item[$\bullet$] if $\beta_{i,j}(J\cap K) > 0$, then $\beta_{i,j}(J) = \beta_{i,j}(K) = 0$.
    \end{enumerate}
    By Theorem \ref{parcon}, this now implies that
    $$\beta_{i,j}(I) = \beta_{i,j}(J)+\beta_{i,j}(K) + 
    \beta_{i-1,j}(J\cap K)$$
    for this fixed pair $(i,j)$.  But since
    this is true for all $(i,j)$ with either
    $i \geq r$ or $j \geq i+s$, this means
    $I=J+K$ is an $(r,s)$-Betti splitting.
\end{proof}

We end this section with consequences for the 
regularity and projective dimension of $I$ 
for a partial Betti splitting.  The
case for a complete Betti splitting
was first shown in \cite[Corollary 2.2]{francisco_splittings_2008}.

\begin{theorem}\label{regprojbounds}
      Suppose that $I=J+K$ is an $(r,s)$-Betti 
      splitting of $I$.  Set
      \begin{eqnarray*}
       m &= &\max\{ {\rm reg}(J), {\rm reg}(K),
      {\rm reg}(J\cap K)-1\}, ~~\mbox{and} \\
      p &=& 
      \max\{ {\rm pd}(I), {\rm pd}(J), {\rm pd}(J\cap K)+1\}.
      \end{eqnarray*}
       Then
      \begin{enumerate}
          \item if $m \geq s$, then 
          ${\rm reg}(I) = m$.
          \item if $p \geq r$, then ${\rm pd}(I) = p$.
        \end{enumerate}
          \end{theorem}

\begin{proof}
    By applying the mapping cone construction to the
    the short exact sequence 
    $$0 \longrightarrow J \cap K \longrightarrow J \oplus K
    \longrightarrow J+K = I \longrightarrow 0,$$
    we always have
    ${\rm reg}(I) \leq m$ and ${\rm pd}(I) \leq p$.

    Since $m \geq s$, this means for all $i \geq 0$ 
    $$\beta_{i,i+m}(I)=\beta_{i,i+m}(J)+\beta_{i,i+m}(K)
    +\beta_{i-1,i+m}(J\cap K)$$
    because we have an $(r,s)$-Betti splitting.
    By our definition of $m$, there is an integer $i$
    such that at least one of the 
    three terms on the right hand side must be nonzero.  This then forces ${\rm reg}(I) \geq m$, thus
    completing the proof that ${\rm reg}(I) = m$.

    Similarly, since $p \geq r$, for all $j \geq 0$
    we have 
    $$\beta_{p,j}(I) = \beta_{p,j}(J)+\beta_{p,j}(K)
    +\beta_{p-1,j}(J\cap K).$$
    By our definition of $p$, there is at least one $j$
    such that one of the terms on the right hand side
    is nonzero, thus showing ${\rm pd}(I) \geq p$.
    Consequently, ${\rm pd}(I) = p$.
\end{proof}

\begin{example}\label{runningexample2}
We illustrate a partial Betti splitting using the 
binomial edge ideal $J_G$ of \Cref{runningexample}.
We ``split'' $J_G$ as $J_G = J + K$ where
\begin{eqnarray*}
    J & = & \langle x_1y_2-x_2y_1, 
    x_1y_3-x_3y_1,
    x_1y_4-x_4y_1,
    x_1y_5-x_5y_1,
    x_1y_7-x_7y_1 \rangle ~~\mbox{and}\\
    K& = & \langle
    x_2y_4-x_4y_2,
    x_2y_5-x_5y_2, 
    x_2y_7-x_7y_2, 
    x_3y_7-x_7y_3,
    x_4y_5-x_5y_4,
    x_6y_7-x_7x_6 \rangle.
\end{eqnarray*}
We compute the graded Betti tables use in \emph{Macaulay2} \cite{mtwo}.
The graded Betti tables of $J$, $K$ and $J \cap K$ are given below.
\footnotesize
\begin{verbatim}
       0  1  2  3 4            0  1  2  3 4             0  1  2  3  4 5
total: 5 20 30 18 4     total: 6 15 20 14 4     total: 15 47 73 62 26 4
    2: 5  .  .  . .         2: 6  2  .  . .         2:  .  .  .  .  . .
    3: . 20 30 18 4         3: . 13  8  . .         3: 10  9  2  .  . .
    4: .  .  .  . .         4: .  . 12 14 4         4:  5 26 21  4  . .
    5: .  .  .  . .         5: .  .  .  . .         5:  . 12 50 58 26 4

    Betti Table J              Betti Table K         Betti Table J intersect K

\end{verbatim}
\normalsize
We compare this to the Betti table of $J_G$:
\footnotesize
\begin{verbatim}
        0  1  2   3  4  5 6  
total: 11 44 89 103 70 26 4
    2: 11 12  3   .  .  . .  
    3:  . 32 62  39  8  . .   
    4:  .  . 24  64 62 26 4

    Betti Table J_G
\end{verbatim}
\normalsize
Then $J_G = J+K$ is {\it not} a complete Betti splitting
since 
$$\beta_{2,4}(J_G) = 3 \neq 
0+ 0+ 9 =\beta_{2,4}(J) + \beta_{2,4}(K) + \beta_{1,4}(
J\cap K).$$
However, this is an example of a $(4,4)$-Betti splitting
since
$$\beta_{i,j}(J_G) =
\beta_{i,j}(J) + \beta_{i,j}(K) +
\beta_{i-1,j}(J\cap K) 
~~\mbox{for all $i \geq 4$ and $j \geq i+4$.}$$
\end{example}

\section{Betti splittings of binomial edge ideals: cut edge case}

In this section and the next, we wish to 
understand when a binomial edge ideal $J_G$ has a (partial) Betti splitting.
A natural candidate to consider is 
when $G_1$ 
is a single edge $e = \{u,v\}$ of $G$ and $G_2 = G\setminus e$.  
More formally, if $f_e = x_uy_v-x_vy_u$ is the binomial associated
to $e$, we wish to understand when
$$J_G = \langle f_e \rangle + J_{G\setminus e}$$
is either a partial or a complete Betti splitting of $J_G$.  As we
show in this section, with some extra hypotheses on $e$, this splitting of 
$J_G$ does indeed 
give a complete Betti splitting.

Since Betti splittings require information 
about the intersection of the two ideals used in the splitting, 
the following lemma shall prove useful.

\begin{lemma}\label{lemma 2.18}
Let $G = (V(G),E(G))$ be a simple graph with $e \in E(G)$. 
Then, using the standard grading of $R$, we have
a graded $R$-module isomorphism
$$[J_{G\setminus e} \cap \langle f_e \rangle]  \cong 
[J_{G\setminus e}: \langle f_e \rangle](-2).$$ 
Furthermore, if $e$ is a cut edge, then 
$$
\beta_{i,j}(J_{(G\setminus e)}\cap \langle f_e\rangle) =
\beta_{i,j-2}(J_{(G\setminus e)_e}) 
~\mbox{for all $i \geq 0$}.$$
\end{lemma}

\begin{proof}
 By definition of quotient ideals, we have that 
 $J_{G\setminus e}: \langle f_e \rangle \xrightarrow{\cdot f_e} J_{(G\symbol{92} e)}\cap \langle f_e\rangle$ is an $R$-module
 isomorphism of degree two.
 This fact implies the first statement.

Now suppose that $e$ is a cut edge. From \Cref{lemma 3.8} we have that $J_{(G\setminus e)_e} = J_{G\setminus e}: \langle f_e \rangle$.
Using this fact and the above isomorphisms of modules, we have
    $$ \tor_i(J_{(G\setminus e)_e},k)_{j-2} = 
    \tor_{i}(J_{G\setminus e}:\langle f_e \rangle, k)_{j-2} \cong
    \tor_{i}(J_{G\setminus e}\cap \langle f_e\rangle, k)_j.
    $$
    This isomorphism imples that 
    $\beta_{i,j}(J_{(G\setminus e)}\cap \langle f_e\rangle) =
    \beta_{i,j-2}(J_{(G\setminus e)_e})$ for all $i \geq 0$ for
    $j \geq 2$.  Now, for any $i \geq 0$ and $j=0$,
    $\beta_{i,0}(J_{(G\setminus e)}\cap \langle f_e\rangle) =
    \beta_{i,0-2}(J_{(G\setminus e)_e}) =0$.  Finally,
    because $J_{(G\setminus e)_e} = 
    J_{G \setminus e} : \langle f_e \rangle$ is generated by
    degree two binomials, then
    $J_{G\setminus e} \cap \langle f_e \rangle$ is generated by degree four
    elements. Thus $\beta_{i,1}(J_{(G\setminus e)}\cap \langle f_e\rangle) =
    \beta_{i,1-2}(J_{(G\setminus e)_e}) =0$ for all $i \geq 0$ and
    $j =1$.
    \end{proof}

With the above lemma, we can study splittings where $e = \{u,v\}$ when
$v$ is a pendant vertex, that is, $\deg v = 1$.

\begin{theorem}\label{maintheo}
    Let $e = \{u,v\} \in E(G)$ with $v$ a pendant vertex. Then 
    \begin{enumerate}
        \item $J_G = J_{G\setminus e}+\langle f_e\rangle$ is a complete Betti splitting, and  
        \item $\beta_{i,j}(J_G) = \beta_{i,j}(J_{G\symbol{92}e}) + \beta_{i-1,j-2}(J_{(G\setminus e)_e})$ for all $i\geq 1$ and
        $j \geq 0$.
    \end{enumerate}
\end{theorem}

\begin{proof}
(1). Let $J_G = \langle f_e\rangle+J_{G\setminus e} \subseteq
R = k[x_1,\ldots,x_n,y_1,\ldots,y_n]$. 
We  consider the $\mathbb{N}^n$-grading on $R$  given by
$\deg x_i = \deg y_i = e_i$, the $i$-th standard basis vector
of $\mathbb{N}^n$.  Note that $J_G$ is a homogeneous ideal with
respect to this grading.

Since $\langle f_e\rangle\cap J_{G\setminus e}\subseteq \langle f_e \rangle$, 
all generators of $\langle f_e\rangle\cap J_{G\setminus e}$ are of the form 
$rf_e = r(x_uy_v-x_vy_u)$, where $r$ is some polynomial in $R$. Hence, the 
multidegree of the generators, and thus the multigraded Betti numbers of the 
ideal $\langle f_e\rangle\cap J_{G\setminus e}$ must occur with
multidegrees $\mathbf{a} = (a_1,\ldots,a_n)$ 
where its $v$-th component $a_v$ is non-zero.  Because
$v$ is a pendant vertex,
$J_{G\setminus e}$ contains no generators having $x_v$ or $y_v$.
Thus, 
$\beta_{i,{\bf a}}(J_{G\symbol{92}e}\cap \langle f_e \rangle )>0$ implies 
that $\beta_{i,{\bf a}}(J_{G \setminus e}) = 0$ for all $i\in \mathbb{N}$ and all multidegrees ${\bf a} \in \mathbb{N}^n$ as defined above.

We have that $\beta_{0,2}(\langle f_e\rangle) = 1$ and 
$\beta_{i,j}(\langle f_e\rangle) = 0$ for $i\neq 0$ and $j\neq 2$ as 
$\langle f_e\rangle$ is a principal ideal. Since 
$J_{G\symbol{92}e}\cap \langle f_e\rangle$ is generated by polynomials 
of degree three or more, this means that 
$\beta_{i,j}(J_{G\symbol{92}e}\cap \langle f_e\rangle)>0$ implies 
$\beta_{i,j}(\langle f_e \rangle) = 0$ 
for all $i\geq  0$ and degrees $j$. It is clear that since this is 
true for all degrees $j$,  this
result also holds for all ${\bf a} \in \mathbb{N}^n$ as well,
that is, 
if $\beta_{i,{\bf a}}(J_{G \setminus e} \cap \langle f_e \rangle) > 0$,
then $\beta_{i,{\bf a}}(\langle f_e \rangle) =0$ for all $i \geq 0$ and 
degrees ${\bf a} \in \mathbb{N}^n$.
    
Therefore \Cref{parcon} implies that 
$$\beta_{i,{\bf a}}(J_G) = \beta_{i,{\bf a}}(J_{G\setminus e})+
\beta_{i,{\bf a}}(\langle f_e \rangle) + 
\beta_{i-1,{\bf a}}(J_{G\setminus e} \cap \langle f_e \rangle)$$
for all $i \geq 0$ and ${\bf a} \in \mathbb{N}^n$.
Since this true for all multidegrees, we can combine them to obtain the same 
result with the degrees $j$ in the standard grading. Hence we have:
$$\beta_{i,j}(J_G) = \beta_{i,j}(\langle f_e\rangle)+
\beta_{i,j}(J_{G\symbol{92} e}) + 
\beta_{i-1,j}(J_{G\symbol{92} e}\cap \langle f_e\rangle) 
~\text{for all $i,j \geq 0$},$$
that is, $J_G = \langle f_e\rangle+J_{G\setminus e}$ is a complete Betti splitting.

An edge with a pendant vertex is a cut edge of $G$. So,
to prove (2), we can 
combine (1) and \Cref{lemma 2.18}
to give
$$\beta_{i,j}(J_G) =  \beta_{i,j}(\langle f_e\rangle)+\beta_{i,j}(J_{G\symbol{92} e}) + \beta_{i-1,j-2}(J_{(G\symbol{92} e)_e})$$ for all 
integers $i \geq 1$ and $j \geq 0$. 
On the other hand, 
$\beta_{i,j}(\langle f_e\rangle) = 0$ for $i\neq 0$ or $j\neq 2$. 
Hence, $\beta_{i,j}(J_G) = \beta_{i,j}(J_{G\symbol{92}e}) + \beta_{i-1,j-2}(J_{(G\symbol{92}e)_e})$ for all $i\geq 1$ and $j \geq 0$.
\end{proof}

In \Cref{maintheo}, we have proved that when there is a cut edge $e$ where one end is a pendant vertex, then removing $e$ induces a complete Betti splitting.  We can now use this result to 
derive complete Betti splittings for more general types of edges.

\begin{theorem}\label{singlefreevertex}
    Let $e = \{u,v\} \in E(G)$ be a cut-edge where $v$ is a free vertex in $G\setminus e$. Then
    \begin{enumerate}
        \item $J_G = J_{G\setminus e}+\langle f_e\rangle$ is a complete Betti splitting, and 
        \item $\beta_{i,j}(J_G) = \beta_{i,j}(J_{G\symbol{92}e}) + \beta_{i-1,j-2}(J_{(G\setminus e)_e})$ for all $i\geq 1$
        and $j \geq 0$.
    \end{enumerate}
\end{theorem}

\begin{proof}
First note that if we can prove $(2)$, then $(1)$ will follow.
To see why, it is immediate that $\beta_{0,j}(J_G) = 
\beta_{0,j}(J_{G\setminus e}) + \beta_{0,j}(\langle f_e \rangle)
+\beta_{-1,j}(J_{G\setminus e} \cap \langle f_e \rangle)$ for 
all $j \geq 0$.  If $i \geq 1$, then 
\begin{eqnarray*}
\beta_{i,j}(J_G) &=& \beta_{i,j}(J_{G\symbol{92}e}) + \beta_{i-1,j-2}(J_{(G\setminus e)_e}) \\
& = & \beta_{i,j}(J_{G\setminus e}) + \beta_{i,j}(\langle f_e \rangle)
+ \beta_{i-1,j}(J_{G \setminus e} \cap \langle f_e \rangle)
\end{eqnarray*}
where we are using \Cref{lemma 2.18} and the fact that
$\beta_{i,j}(\langle f_e \rangle) = 0$ for all $i \geq 1$.

Now note that to prove to $(2)$, we can pass to quotient rings 
and prove that 
$$\beta_{i,j}(R/J_G) = \beta_{i,j}(R/J_{G\setminus e}) 
+ \beta_{i-1,j-2}(R/J_{(G\setminus e)_e} )
 ~~\mbox{for all $i \geq 2$ and $j \geq 0$}.$$

Let $G$ be a connected graph with cut-edge $e = \{u,v\}$. 
Let $G_1$ and $G_2$ be the connected components of $G\setminus e$,
and suppose $u\in V(G_1)$ and $v\in V(G_2)$. 
By our hypotheses, the vertex $v$ is a free vertex in $G_2$. 
Hence, we can see that $G$ is a decomposable graph, with 
decomposition $G = (G_1\cup \{e\}) \cup_v G_2$ 
(since pendant vertices are free vertices and $v$ is a pendant vertex of $e$).

By \Cref{freevertexbetti} we have
    \begin{equation}\label{5.21}
        \beta_{i,j}(R/J_G) = \sum_{\substack{0 \leq i_1\leq i \\ ~0 \leq j_1\leq j}}\beta_{i_1,j_1}(R/J_{G_1\cup \{e\}})\beta_{i-i_1,j-j_1}(R/{J_{G_2}}).
    \end{equation}
    Since $e$ is a cut-edge with a pendant vertex in $G_1 \cup \{e\}$, we can now apply \Cref{maintheo}
    to $R/J_{G_1 \cup \{e_1\}}$. Thus,
    \begin{multline}\label{1.2}
        \sum_{\substack{0 \leq i_1\leq i \\0 \leq j_1\leq j}}\beta_{i_1,j_1}(R/{J_{G_1\cup \{e\}}})\beta_{i-i_1,j-j_1}(R/{J_{G_2}}) = \\ 
        \sum_{\substack{2\leq i_1\leq i \\ 0 \leq j_1\leq j}}(\beta_{i_1,j_1}(R/{J_{G_1}}) + \beta_{i_1-1,j_1-2}(R/{J_{(G_1)_e}}))\beta_{i-i_1,j-j_1}(R/{J_{G_2}}) 
        \\
        + (\beta_{1,2}(R/{J_{G_1}})+ 1)\beta_{i-1,j-2}(R/{J_{G_2}}) + \beta_{i,j}(R/{J_{G_2}}).
    \end{multline}
   
    Here, we are using the fact
    that $\beta_{1,j}(R/J_{G_1 \cup \{e\}}) =0$ 
    if $j \neq 2$, and when $j=2$,
    $J_{G_1 \cup \{e\}}$ has one more generator
    than $J_{G_1}$, that is, $\beta_{1,2}(R/J_{G_1 \cup \{e\}})
    = \beta_{1,2}(R/J_{G_1})+1$.
    
By expanding out and regrouping, we get
\footnotesize
        \begin{align} \label{1.3}
        \beta_{i,j}(J_G) =&  \sum_{
        \substack{1\leq i_1\leq i \\
        0\leq j_1\leq j}}\beta_{i_1,j_1}(R/{J_{G_1}})\beta_{i-i_1,j-j_1}(R/{J_{G_2}}) + \beta_{i,j}(R/{J_{G_2}}) \nonumber\\
        & + \sum_{\substack{2\leq i_1\leq i \\
        0 \leq j_1\leq j}}\beta_{i_1-1,j_1-2}(R/{J_{(G_1)_e}})\beta_{i-i_1,j-j_1}(R/{J_{G_2}}) +\beta_{i-1,j-2}(R/{J_{G_2}}) \nonumber\\
        =&   \sum_{
        \substack{0 \leq i_1\leq i \\
        0 \leq j_1\leq j}}\beta_{i_1,j_1}(R/{J_{G_1}})\beta_{i-i_1,j-j_1}(R/{J_{G_2}})+  \sum_{\substack{0 \leq i_1\leq i-1 
        \\ 0 \leq j_1\leq j-2}}\beta_{i_1,j_1}(R/{J_{(G_1)_e}})\beta_{i-1-i_1,j-2-j_1}(R/{J_{G_2}}).
    \end{align}
    \normalsize
Since $G_1$ and $G_2$ are graphs on disjoint sets of vertices, $J_{G_1}$ and $J_{G_2}$ are ideals on disjoint sets of variables. Hence, 
\begin{align}\label{1.4}
    \sum_{\substack{0\leq i_1\leq i \\
    0\leq j_1\leq j}}\beta_{i_1,j_1}(R/{J_{G_1}})\beta_{i-i_1,j-j_1}(R/{J_{G_2}}) & = \beta_{i,j}(R/{J_{G_1}+J_{G_2}}) \nonumber \\
    &=\beta_{i,j}(R/{J_{G_1\cup G_2}}) = \beta_{i,j}(R/{J_{(G\setminus e)}}).
\end{align}
Similarly, the same is true for $(G_1)_e$ and $G_2$. Note, that since $v$ is already a free vertex of $G_2$, we have $(G\setminus e)_e = (G_1)_e \cup G_2$. Hence,
\begin{align}\label{1.5}
    \sum_{\substack{0 \leq i_1\leq i-1 \\ 0 \leq j_1\leq j-2}}\beta_{i_1,j_1}(R/{J_{(G_1)_e}})\beta_{i-1-i_1,j-2-j_1}(R/{J_{G_2}}) & = \beta_{i-1,j-2}(R/{J_{(G_1)_e}+J_{G_2}}) \nonumber\\
    & = \beta_{i-1,j-2}(R/{J_{(G_1)_e\cup G_2}})
    \nonumber \\
    & = \beta_{i-1,j-2}(R/{J_{(G\setminus e)_e}}).
\end{align}
Thus, substituting \Cref{1.5} with \Cref{1.4}
into \Cref{1.3}, we get the desired conclusion.
\end{proof}

Because we have a complete Betti splitting,
\Cref{regprojbounds} implies the collorary.

\begin{corollary}\label{singlevertexcor}
    With the hypotheses as in \Cref{singlefreevertex},
    \begin{eqnarray*}
     {\rm reg}(J_G) &= &\max\{{\rm reg}(J_{G\setminus e}),
    {\rm reg}((J_{G \setminus e})_e) +1\} 
     ~~\mbox{and} \\
     {\rm pd}(J_G) &= &\max\{{\rm pd}(J_{G\setminus e}), 
     {\rm pd}(J_{(G \setminus e)_e}) +1\}. 
     \end{eqnarray*}
\end{corollary}

\begin{proof}
    Because $J_G = J_{G\setminus e} + \langle f_e \rangle$
    is a complete Betti splitting, \Cref{regprojbounds} gives
    \begin{eqnarray*}
     {\rm reg}(J_G) &= &\max\{{\rm reg}(J_{G\setminus e}),
     {\rm reg}(\langle f_e \rangle), 
     {\rm reg}(J_{G \setminus e} \cap \langle f_e \rangle) -1\} 
     ~~\mbox{and} \\
     {\rm pd}(J_G) &= &\max\{{\rm pd}(J_{G\setminus e}),
     {\rm pd}(\langle f_e \rangle), 
     {\rm pd}(J_{G \setminus e} \cap \langle f_e \rangle) +1\}. 
     \end{eqnarray*}
The result now follows since
$2 = {\rm reg}(\langle f_e \rangle) \leq 
{\rm reg}(J_{G \setminus e})$ and 
$0 = {\rm pd}(\langle f_e \rangle)$ and because
\Cref{lemma 2.18} implies 
 ${\rm reg}(J_{G \setminus e} \cap \langle f_e \rangle) 
    = {\rm reg}(J_{(G\setminus e)_e}) +2$
    and 
    ${\rm pd}(J_{G \setminus e} \cap \langle f_e \rangle) = {\rm pd}(J_{(G \setminus e)_e})$.
\end{proof}

Recall that an edge $e = \{u,v\}$ is a free cut-edge of $G$
if both $u$ and $v$ are free vertices of $G \setminus e$.  
When \Cref{singlefreevertex} is applied to a free cut-edge,
we can recover the following results 
of  Saeedi Madani and Kiani \cite{kiani_regularity_2013-1}.

\begin{corollary}[{\cite[Proposition 3.4]{kiani_regularity_2013-1}}]
\label{freecutedge}
Let $e = \{u,v\} \in E(G)$ be a free cut-edge.
Then
    \begin{enumerate}
        \item $\beta_{i,j}(J_G) = \beta_{i,j}(J_{G\setminus e}) + \beta_{i-1,j-2}(J_{G\setminus e})$,
        \item \rm pd($J_G$) = pd($J_{G\setminus e}) + 1$, and
        \item \rm reg($J_G$) = reg($J_{G\setminus e}$) + 1.
    \end{enumerate}
\end{corollary}

\begin{proof}
    When $e$ is a free cut-edge of $G$, then
    $(G\setminus e)_e = G\setminus e$.  The results then
    follow from \Cref{singlefreevertex} and \Cref{singlevertexcor} by using the equality
    $J_{(G\setminus e)_e} = J_{G\setminus e}.$
\end{proof}

One application of \Cref{maintheo} is finding the Betti numbers of the binomial edge ideals of certain graphs. The corollary below 
is a new proof
of \cite[Proposition 3.8]{jayanthan_almost_2021}
for the graded Betti
numbers of the binomial edge ideals of any
star graph $S_n$.

\begin{corollary}\label{star}
    Let $S_n$ denote the star graph on $n$-vertices. Then we have:
    \[ \beta_{i}(J_{S_n}) = 
    \beta_{i,i+3}(J_{S_n}) = i\binom{n}{i+2} \text{\hspace{4mm} $i\geq 1$}. \]
    Furthermore, $\beta_0(J_{S_n}) = 
    \beta_{0,2}(S_n) = n-1$.
\end{corollary}

\begin{proof}
     Note that the statement about $0$-th
     graded Betti numbers just follows from 
     the fact 
     that $S_n$ has $n-1$ edges.

Consider the edge $e =\{1,n\}$. Since $S_n\setminus e =  S_{n-1} \cup \{n\}$, 
we have $(S_n\setminus e)_e = K_{n-1} \cup \{n\}$. So from \Cref{maintheo}, we have:
    \[\beta_{i,j}(J_{S_n}) = \beta_{i,j}(J_{S_{n-1}})+\beta_{k-1,j-2}(J_{K_{n-1}})
    ~~\text{ for all $i\geq 1$}.\]
    We can now use induction to show the above assertion. For $n = 2$, we can see that $S_2$ is just an edge. We know that $\beta_{i,j}(J_{S_2}) = 0$ for all $i\geq 1$. Hence, we can see that it agrees with the above formula as $\binom{2}{r} = 0$ when $r>2$. Now assume the formula holds for $n-1$. We must show that it holds for $n$.

    From \Cref{completebetti}, we know that $\beta_{i,i+2}(J_{K_{n-1}}) = (i+1)\binom{n-1}{i+2}$ and $\beta_{i,j}(J_{K_{n-1}}) = 0$ 
    if $j\neq i+2$. 
        Hence, using induction and \Cref{maintheo}, we can see that $\beta_{i,j}(J_{S_n}) = \beta_{i,j}(J_{S_{n-1}})+\beta_{i-1,j-2}(J_{K_{n-1}})=0+0$, when $j\neq i+3$. We also 
        have 
    \[\beta_{i,i+3}(J_{S_n}) = \beta_{i,i+3}(J_{S_{n-1}})+\beta_{i-1,i+1}(J_{K_{n-1}}) = i\binom{n-1}{i+2}+i\binom{n-1}{i+1} = i\binom{n}{i+2}.\]
    This verifies the formula of the statement.
\end{proof}


\section{Partial Betti splittings of binomial edge ideals: 
\texorpdfstring{$s$}{s}-partitions}

In this section we consider the other natural candidate to study in the context of partial Betti splittings.  In this case, we fix a vertex $s \in V)$, and let $G_1$ be the graph with $E(G_1)$ equal to the set of edges of  $G$ that contain $s$ (so $G_1$ is isomorphic
to a star graph) and $G_2 = G \setminus \{s\}$.  
We formalize this idea in the next definition.

\begin{definition}\label{vpart}
For $s\in V(G)$, an {\it $s$-partition} of $J_G$ is the splitting 
$J_G = J_{G_1}+J_{G_2},$ 
where $G_1$ is the subgraph of $G$ with $V(G_1) = N_G[s]$ and  
$E(G_1) = \{\{s,k\}\mid k\in N_G(s)\}$, and $G_2=G\setminus \{s\}$.
\end{definition}

Note that the graph $G_1$ in an $s$-partition is isomorphic to the star graph $S_{\deg(s)+1}$.  We will show that
an $s$-partition always gives a partial Betti splitting of $J_G$:

\begin{theorem}\label{maintheo2}
Let $G$ be a graph on $[n]$ and let $J_G = J_{G_1}+J_{G_2}$ 
be an $s$-partition of $G$ for some $s\in [n]$. Let $c(s)$ 
be the size of the largest clique containing $s$. Then, for all 
$i, j$ with $i \geq c(s)$ or $j \geq i+4$,
    \begin{equation*}
        \beta_{i,j}(J_G) = \beta_{i,j}(J_{G_1})+\beta_{i,j}(J_{G_2})+\beta_{i-1, j}(J_{G_1}\cap J_{G_2}).
    \end{equation*}
In other words, $J_G = J_{G_1}+J_{G_2}$ is a $(c(s), 4)$-Betti splitting.
\end{theorem}

Our proof hinges on a careful examination
of $J_{G_2} \cap J_{G_2}$, which is carried out below.

\begin{lemma}\label{deg3gen}
Let $G$ be a graph on $[n]$ and let $J_G = J_{G_1}+J_{G_2}$ be an $s$-partition of $G$ for some $s\in [n]$.
Then the set
\[
\mathcal{B} = \{x_sf_{a,b}, y_sf_{a,b}\mid a,b\in N_G(s) \text{ and } \{a,b\}\in E(G)\}.\]
is a $k$-basis for $(J_{G_1} \cap J_{G_2})_3$.
\end{lemma}

\begin{proof}
Let $N_G(s) = \{v_1,\dots, v_r\}$. Since $E(G_1) \cap E(G_2) = \emptyset$, the generators of $J_{G_1} \cap J_{G_2}$ are of degree at least $3$. 
First of all observe that 
$\B_1 = \{x_if_e, y_if_e\mid e \in E(G_1) \text{ and } i\in \{1, \dots, n\}\}$ and 
$\B_2 = \{x_if_e, y_if_e\mid e\in E(J_{G_2}) \text{ and } i\in \{1, \dots, n\}\}$ form $k$-bases for the subspaces $(J_{G_1})_3$ and $(J_{G_2})_3$ respectively.
    
Let $P \in (J_{G_1} \cap J_{G_2})_3 = (J_{G_1})_3 \cap (J_{G_2})_3$. Write 
\begin{equation}\label{eq.P}
P = \sum_{g_{i,e}\in \B_1}c_{i,e} g_{i,e},    
\end{equation}
where $c_{i,e} \in k$. 

We first claim that the coefficients of $x_if_{a,s}$ and $y_if_{a,s}$ in the linear combination of $P$ are zero if $i \notin \{v_1,\ldots, v_r\}$. We prove this for $x_if_{a,s}$ and the other proof is similar. 
Let $c$ be the coefficient of $x_if_{a,s}$.
Observe that, since $i\notin \{v_1,\dots, v_k\}$, the term  $y_sx_ix_a$ in $P$, appears in only one basis element, namely $x_if_{a,s}$. 
Since $P$ is in $(J_{G_2})_3$ as well, we can write
\begin{equation}\label{2.8}
    P = S+ y_s(c x_ix_a+L) = Q + y_s\left(\sum_{f_e\in \mathfrak{G}(J_{G_2})}c'_e f_e\right),
\end{equation}
where no terms of $S$ and $Q$ are divisible by $y_s$ and $L$ does not have any monomial terms divisible by $x_ix_a$. Since $y_s$ does not divide any term of $S$ and $Q$, the above equality implies that $c x_ix_a+L = \sum_{f_e\in \mathfrak{G}(J_{G_2})}c'_e f_e$. Now by considering the grading on $R$ given by $\deg x_j = (1,0)$ and $\deg y_j = (0,1)$ for all $j$, we can see that $x_ix_a$ is of degree $(2,0)$ but the degree of each term $f_e$ in $\mathfrak{G}(J_{G_2})$ is $(1,1)$. Hence, for \Cref{2.8} to hold, $c=0$. This completes the proof of the claim. 

Now consider the case where $i\in \{v_1,\dots, v_k\}$. In this case, it can be seen that the term $y_sx_ix_a$ when written as an element of $(J_{G_1})_3$ appears in the basis elements $x_if_{a,s}$ and $x_af_{i,s}$, and in no other basis element. As before, to make sure that there are no elements of degree $(2,0)$, the coefficients of $x_if_{a,v}$ and $x_af_{i,s}$ in \Cref{eq.P} must be additive inverses of each other. Denote the coefficient of $x_if_{a,s}$ by $c$. Then, 
$$cx_if_{a,s} - cx_af_{i,s} = cx_s(x_ay_i-x_iy_a) = cx_sf_{a,i}.$$ Similar arguments show that the coefficients of $y_if_{a,s}$ and $y_af_{i,s}$ must be additive inverses of each other, and that the corresponding linear combination in the \Cref{eq.P} appears as $c'y_sf_{a,i}$. Therefore, \Cref{eq.P} becomes 
    \[P = 
    \sum_{a,i\in N_G(s)}c_{i,a} x_sf_{a,i}+c'_{i,a} y_sf_{a,i}.\]
Since $P \in (J_{G_2})_3$, it is easily observed that $c_{i,a} = 0$ whenever $\{i,a\} \notin E(G)$. Therefore, $\mathcal{B}$ spans the subspace $(J_{G_1} \cap J_{G_2})_3$. Linear independence is fairly straightforward as $s \neq a, b$ for any $a, b \in N_G(s)$. Hence the assertion of the lemma is proved.
\end{proof}

\begin{remark}\label{deg4}
If $G$ is a triangle-free graph, then there does not 
exist any $a,b\in N_G(s)$ with $\{a,b\}\in E(G)$ for 
any $s\in V(G)$. Hence it follows from \Cref{deg3gen} 
that there are no degree 3 generators of 
$J_{G_1}\cap J_{G_2}$ for any $s$-partition. Hence, 
$J_{G_1} \cap J_{G_2}$ is generated by elements 
of degrees $4$ or higher.
\end{remark}

Since the generators of $J_{G_1}\cap J_{G_2}$ resemble the generators of a binomial edge ideal, we can calculate its linear strand in terms of the linear strand of some binomial edge ideal.

\begin{theorem}\label{thm:Betti-intersection}
Let $G$ be a graph on $[n]$ and let $J_G = J_{G_1}+J_{G_2}$ be an $s$-partition of $G$ for some $s\in [n]$.  If 
$G'$ is the induced subgraph of $G$ on $N_G(s)$, then
    \[\beta_{i,i+3}(J_{G_1}\cap J_{G_2}) = 2\beta_{i,i+2}(J_{G'})+\beta_{i-1,i+1}(J_{G'})\text{\hspace{2mm} for all $i\geq 0$}.\]
\end{theorem}

\begin{proof}
From \Cref{deg3gen}, we have that the minimal degree 3 generators for $J_{G_1}\cap J_{G_2}$ are
    \[L  =\{x_sf_{a,b}, y_sf_{a,b}\mid a,b\in N_G(s) \text{ and } \{a,b\}\in E(G)\}.\]
Since, $J_{G_1}\cap J_{G_2}$ is generated in degree 3 or higher, if $I$ is the ideal generated by $L$, 
then $\beta_{i,i+3}(J_{G_1}\cap J_{G_2}) = \beta_{i,i+3}(I)$ for all $i \geq 0$.
Now consider the partition $I = I_x+I_y$, where
$$
\mathfrak{G}(I_x) = \{x_sf_{a,b}\mid \text{ $\{a,b\}\in E(G')$}\} ~\mbox{and} ~
\mathfrak{G}(I_y) = \{y_sf_{a,b}\mid \text{$\{a,b\}\in E(G')$}\}.
$$

We now claim that
        \[I_x\cap I_y = \langle\{x_sy_sf_{a,b}\mid \text{$\{a,b\}\in E(G')$}\}\rangle.\]
It is clear that each $x_sy_sf_{a,b} \in I_x\cap I_y$. For the other inclusion, consider $g\in I_x\cap I_y$. Since $g$ is in both $I_x$ and $I_y$, we can write $g$ as
        \[g = x_s\left(\sum k_{a,b}f_{a,b}\right) = 
        y_s\left(\sum k'_{a,b}f_{a,b}\right),\]
where $k_{a,b}, k'_{a,b} \in R$. Since, none of the $f_{a,b}$'s involve the variables $x_s$ and $y_s$, some terms of $k_{a,b}$ are divisible by $y_s$, for each $\{a,b\}\in E(G')$. Separating out the terms which are divisible by $y_s$, write:
    \[g  = x_s\left(\sum k_{a,b}f_{a,b}\right) = 
    x_s\left(\sum y_sh_{a,b}f_{a,b}+L\right),\]
where no term of $L$ is divisible by $y_s$. Since $g$ is divisible by $y_s$, we have that $y_s|L$. But since no term of $L$ is divisible by $y_s$, this implies that $L=0$. Hence, 
$$g = x_sy_s\left(\sum h_{a,b}f_{a,b}\right)\in \langle\{x_sy_sf_{a,b}\mid \text{$\{a,b\}\in E(G')$}\}\rangle.$$
        
It is readily seen that $J_{G'}\xrightarrow{\cdot x_s} I_x$, $J_{G'}\xrightarrow{\cdot y_s} I_y$, and $J_{G'}\xrightarrow{\cdot x_sy_s} I_x\cap I_y$ are isomorphisms of degree 1, 1, and 2 respectively.  Now, consider $\mathbb{N}^n$ multigrading on $R$ with $\deg x_i = \deg y_i =  e_i$ for all $i=1,\ldots, n$. The above isomorphisms imply that:
    \[\tor_i(I_x,k)_{\mathbf{a}+e_s}\cong \tor_i(J_{G'},k)_{\mathbf{a}} \cong \tor_i(I_y,k)_{\mathbf{a}+e_s} \]
    and
    $$\tor_i(I_x\cap I_y,k)_{\mathbf{a}+2e_s}\cong \tor_i(J_{G'},k)_{\mathbf{a}},$$
where $\mathbf{a} = (a_1,\ldots,a_n) \in \mathbb{N}^n$ with $a_s=0$. Summing up all the multigraded Betti numbers, we get $\beta_{i,j}(I_x) = \beta_{i,j-1}(J_{G'}) = \beta_{i,j}(I_y) $ and $\beta_{i,j}(I_x\cap I_y) = \beta_{i,j-2}(J_{G'})$.  Observe that all the non-zero
multigraded Betti numbers of $I_x\cap I_y$ occur only on 
multidegrees $\mathbf{a}+2e_s$ while all 
Betti numbers of $I_x$ and $I_y$ occur only at $\mathbf{a}+e_s$. Hence, by using \Cref{parcon}  and combining all multidegrees, we have 
$$\beta_{i,j}(I) = \beta_{i,j}(I_x)+\beta_{i,j}(I_y)+\beta_{i-1,j}(I_x\cap I_y) ~~\mbox{for all $i,j \geq 0$}.$$ 
Therefore,
    \[\beta_{i,i+3}(J_{G_1}\cap J_{G_2}) = \beta_{i,i+3}(I) = \beta_{i,i+2}(J_{G'})+\beta_{i,i+2}(J_{G'})+\beta_{i-1,i+1}(J_{G'})\]
for all $i \geq 0$.
\end{proof}

We can now prove the main result of this section:

\begin{proof}[Proof of \Cref{maintheo2}]
We first prove that $\beta_{i,i+3}(J_{G_1}\cap J_{G_2}) = 0$ for all $i\geq c(s)-1$, since we will require this fact later in the proof.
It follows from \Cref{thm:Betti-intersection} that for all $i \geq 0$ 
    \[\beta_{i,i+3}(J_{G_1}\cap J_{G_2}) = 2\beta_{i,i+2}(J_{G'})+\beta_{i-1,i+1}(J_{G'}),\] 
where $G'$ is the induced subgraph of $G$ on $N_G(s)$. From 
\Cref{linearbinom}, we get $\beta_{i,i+2}(J_{G'}) = (i+1)f_{i+1}
(\Delta(G'))$, where $f_k(\Delta(G'))$ is the number of faces of 
$\Delta(G')$ of dimension $k$. Since the largest clique in $G'$ is of size $c(s)-1$,  $\beta_{i,i+2}(J_{G'}) = 0$ for all $i\geq c(s)-2$. Hence $\beta_{i,i+3}(J_{G_1}\cap J_{G_2}) = 0$ for all $i\geq c(s)-1$ by
the above formula.

Consider the $\mathbb{N}^n$-grading on $R$ 
given by $\deg x_i = \deg y_i = e_i$, the $i$-th unit vector.   
Now
fix any $i \geq 1$ and let 
${\bf a} = (a_1,\ldots,a_n) \in \mathbb{N}^n$ 
with $\sum_{\ell=1}^n a_\ell  \geq i+ 4$.  
All the generators of $J_{G_1}\cap J_{G_2}$ are of the form $fx_s+gy_s$, 
so their multigraded Betti numbers occur within multidegrees 
$\mathbf{a}$ such that its $s$-th component, $a_s$ is non-zero.  Since $J_{G_2}$ contains no generators of the form 
$fx_s+gy_s$, $\beta_{i,{\bf a}}(J_{G_1}\cap J_{G_2})>0$ implies that $\beta_{i,{\bf a}}(J_{G_2}) = 0$ for all $i\in \mathbb{N}$,
and similarly, $\beta_{i-1,{\bf a}}(J_{G_1} \cap J_{G_2}) > 0$
implies that $\beta_{i,{\bf a}}(J_{G_2}) = 0$

From \Cref{star}, since $G_1$ is a star graph,
    \[ \beta_{i}(J_{G_1}) = \beta_{i,i+3}(J_{G_1}) = i\binom{\deg(s)}{i+2} 
    ~\mbox{for all $i\geq 1$}.\]
    Hence, we can see that for all multidegrees ${\bf a} = 
    (a_1,\dots,a_n)$ with $\sum_{\ell=1}^n a_\ell\geq i+4$, we also have
$\beta_{i,{\bf a}}(J_{G_1}\cap J_{G_2})>0$ implies that 
$\beta_{i,{\bf a}}(J_{G_1})=0$, and $\beta_{i-1,{\bf a}}(J_{G_1}\cap J_{G_2})>0$ implies that $\beta_{i-1,{\bf a}}(J_{G_1})=0$.

Therefore, from \Cref{parcon}, we have 
    \[\beta_{i,{\bf a}}(J_G) = \beta_{i,{\bf a}}(J_{G_1})+
    \beta_{i,{\bf a}}(J_{G_2})+
    \beta_{i-1, {\bf a}}(J_{G_1}\cap J_{G_2}),\] 
    for all $i \geq 0$ and multidegrees ${\bf a}$ with $\sum_{\ell=1}^n a_\ell\geq i+4$.

Now fix any $i \geq c(s)$ and ${\bf a} \in \mathbb{N}^n$.  As 
argued above, if $\beta_{i,{\bf a}}(J_{G_1} \cap J_{G_2})>0$, then
$\beta_{i,{\bf a}}(J_{G_2}) = 0$ (and a similar statement for 
$\beta_{i-1,{\bf a}}(J_{G_1} \cap J_{G_2})$). 
We also know that if $\beta_{i,{\bf a}}(J_{G_1} \cap J_{G_2}) > 0$,
with $i \geq c(s)-1$, then $\sum_{\ell=1}^n a_l \geq i+4$ since
$J_{G_1} \cap J_{G_2}$ is generated in degree three and 
$\beta_{i,i+3}(J_{G_1}\cap J_{G_2}) =0$ for all $i \geq c(s)-1$.
On the other hand, since ${\rm reg}(J_2) = 3$ by \Cref{star},
we have $\beta_{i,{\bf a}}(J_{G_2}) = 0$ for all 
$\sum_{\ell=1}^n a_\ell \neq i+3$ if $i \geq 1$.
So, we have shown that if 
$\beta_{i,{\bf a}}(J_{G_1} \cap J_{G_2}) > 0$, then
$\beta_{i,{\bf a}}(J_{G_2}) = 0$, and also if
$\beta_{i-1,{\bf a}}(J_{G_1} \cap J_{G_2}) > 0$, then $\beta_{i-1,{\bf a}}(J_{G_2}) = 0$.

So by using \Cref{parcon}, we  have 
    \[\beta_{i,{\bf a}}(J_G) = \beta_{i,{\bf a}}(J_{G_1})+
    \beta_{i,{\bf a}}(J_{G_2})+
    \beta_{i-1, {\bf a}}(J_{G_1}\cap J_{G_2}),\] 
    for all $i \geq c(s)$ and multidegrees ${\bf a} \in \mathbb{N}^n$.

    Therefore, by combining these two results we have 
    \[\beta_{i,{\bf a}}(J_G) = \beta_{i,{\bf a}}(J_{G_1})+
    \beta_{i,{\bf a}}(J_{G_2})+
    \beta_{i-1,{\bf a}}(J_{G_1}\cap J_{G_2}),\] 
    for all $i$ and multidegrees ${\bf a}$ with $i\geq c(s)$ or $\sum_{k=1}^n a_k\geq i+4$.  By summing over all multidegrees,
    we obtain the same result for the standard grading, i.e.,
    $$\beta_{i,j}(J_G) = \beta_{i,j}(J_{G_1})+
    \beta_{i,j}(J_{G_2})+
    \beta_{i-1, j}(J_{G_1}\cap J_{G_2}),$$ 
    for all $i,j$  with $i\geq c(s)$ or $j\geq i+4$. In
    other words, we have a $(c(s),4)$-Betti splitting.
\end{proof}

\begin{example}
If $G$ is the graph of \Cref{runningexample}, then we saw in
\Cref{runningexample2} that the ideal $J_G$ has a 
$(4,4)$-Betti splitting.  Note that the splitting
of \Cref{runningexample2} is an example of an $s$-partition
with $s=1$.  Furthermore, the largest clique that the
vertex $s=1$ belongs to has size four (there is a clique on the
vertices $\{1,2,4,5\})$.  So, by the previous
result $J_G$ will have a $(c(1),4)$-Betti splitting with
$c(1)=4$, as shown in this example.
\end{example}

\begin{corollary}\label{trianglefree}
Let $G$ be a graph on $[n]$ and let $J_G = J_{G_1}+J_{G_2}$ be an $s$-partition of $G$ for some $s\in [n]$. 
If $G$ is a triangle-free graph, then
$J_G = J_{G_1}+J_{G_2}$ is a complete Betti splitting.
\end{corollary}

\begin{proof}
Since $G$ is a triangle-free graph, the largest clique containing $s$
is a $K_2$, i.e., $c(s)=2$.  Thus  \Cref{maintheo2} implies that
$J_G = J_{G_1}+J_{G_2}$ is a $(2,4)$-Betti splitting, that
is, 
$$\beta_{i,j}(J_G) = \beta_{i,j}(J_{G_1})+\beta_{i,j}(J_{G_2})+\beta_{i-1, j}(J_{G_1}\cap J_{G_2} )\text{ for all $i\geq 2$ or $j \geq i +4$.}$$

To complete the proof, we just need to show the
above formula also holds for the graded Betti numbers
$\beta_{i,j}(J_G)$ with $(i,j) \in \{(0,0),(0,1),(0,2),(0,3),(1,1),
(1,2),(1,3),(1,4)\}$.
We always have $\beta_{0,j}(J_G) = 
\beta_{0,j}(J_{G_1})+\beta_{0,j}(J_G) + \beta_{-1,j}(J_{G_1}\cap J_{G_2})$ for all $j \geq 0$.  Also, since $J_G, J_{G_1}$ and 
$J_{G_2}$ are generated in
degree $2$ and $J_{G_1} \cap J_{G_2}$ generated in degree
four (by \Cref{deg4}), we have  
$$0 = \beta_{1,j}(J_G) = \beta_{1,j}(J_{G_1})+\beta_{1,j}(J_G) + \beta_{0,j}(J_{G_1}\cap J_{G_2}) = 0 + 0 + 0$$
for $j=1,2$.  

Finally, because $J_{G_1} \cap J_{G_2}$ is generated
in degree four, we have $\beta_{1,3}(J_{G_1}\cap J_{G_2}) = 
\beta_{1,4}(J_{G_1}\cap J_{G_2}) = 0$.  Thus, for $(i,j) = (1,3)$
the conditions of \Cref{parcon} are vacuously satisfied
(since $\beta_{1,3}(J_{G_1}\cap J_{G_2}) = \beta_{0,3}(J_{G_1}\cap J_{G_2}) = 0$).  For
$i=1$ and $j=4$, we have $\beta_{1,4}(J_{G_1}\cap J_{G_2}) = 0$ and
when $\beta_{0,4}(J_{G_1} \cap J_{G_2}) > 0$, we have $\beta_{0,4}(J_{G_1}) = \beta_{0,4}(J_{G_2}) =0$
since both $J_{G_1}$ and $J_{G_2}$ are generated in degree 2.   So
again the conditions of \Cref{parcon} are satisfied. Thus
$$ \beta_{1,j}(J_G) = \beta_{1,j}(J_{G_1})+\beta_{1,j}(J_{G_2}) + \beta_{1,j}(J_{G_1}\cap J_{G_2}) = \beta_{1,j}(J_{G_1})+\beta_{1,j}(J_G) $$
for $j=3,4$.
\end{proof}

\begin{corollary}
Let $G$ be a graph on $[n]$ and let $J_G = J_{G_1}+J_{G_2}$ be an $s$-partition of $G$ for some $s\in [n]$. 
    \begin{enumerate}
        \item If $\pd(J_G)\geq c(s)$, then
     $\pd(J_G) = \max\{ \pd(J_{G_1}), \pd(J_{G_2}), \pd(J_{G_1}\cap J_{G_2})+1\}.$

        \item If $\reg(J_G)\geq 4$, then
        $\reg(J_G) = \max\{\reg(J_{G_2}), \reg(J_{G_1}\cap J_{G_2})-1\}.$
    \end{enumerate}

\end{corollary}

\begin{proof}
    Given that $\pd(J_G)\geq c(s)$, we know that there is a partial splitting for all $\beta_{i,j}(J_G)$, for all $i\geq c(s)$. Hence, $\pd(J_G) = \max\{ \pd(J_{G_1}), \pd(J_{G_2}), \pd(J_{G_1}\cap J_{G_2})+1\}$.

    Similarly, if $\reg(J_G)\geq 4$, we know that there is a partial splitting for all $\beta_{i,j}(J_G)$, for all $i\geq c(s)$. Hence, $\reg(J_G) = \max\{ \reg(J_{G_1}), \reg(J_{G_2}), \reg(J_{G_1}\cap J_{G_2})-1\}$. Since $\reg(J_{G_1}) = 3$, we have $\reg(J_G) = \max\{\reg(J_{G_2}), \reg(J_{G_1}\cap J_{G_2})-1\}$.
\end{proof}


\section{On the total Betti numbers of binomial edge ideals of trees}

In this section, we explore an application of \Cref{maintheo} to find certain Betti numbers of trees. In particular, we obtain a precise expression for the second Betti number of $J_T$ for any tree $T$. Note that $\beta_1(J_T)$ was first computed in \cite[ Theorem 3.1]{jayanthan_almost_2021}. We begin with recalling a simple technical result that we require in our main results.   

\begin{lemma}\label{pendantexist}
Let $T$ be a tree which is not an edge with $v\in V(T)$ and let $S_v = \{u\in N_T(v) ~|~ \deg u > 1\}$. Then, there exists $a\in V(T)$ with $\deg a>1$ such that
    $|S_a|\leq 1.$
\end{lemma}

\begin{proof}
    See \cite[Proposition 4.1]{JK2005}.
\end{proof}

To compute the second Betti number of $J_T$, we use \Cref{maintheo} to reduce the computation to graphs with a fewer number of vertices. 
One of the graphs involved in this process becomes a clique sum of a tree and a complete graph. So, we now compute the first Betti number of this class of graphs.

\begin{theorem}\label{T+K_m}
Let $G=T \cup_{a} K_m$. If $|V(G)| = n$, then 
    \begin{eqnarray*}
        \beta_1(J_G) &= &\binom{n-1}{2}+2\binom{m}{3}+\sum_{w\notin V(K_m)}\binom{\deg_G w}{3}+\binom{\deg_G  a-m+1}{3} \\ & &+(n-m-1)\binom{m-1}{2}
        +(m-1)\binom{\deg_G a -m+1}{2}.
    \end{eqnarray*}
\end{theorem}

\begin{proof}
We prove the assertion by induction on $|V(T)|$. If $|V(T)| = 1$, then $G$ is a complete graph and $n = m$. Therefore, by \Cref{completebetti}
    \[\beta_1(J_G) = 2\binom{n}{3} = \binom{n-1}{2}+2\binom{n}{3}-\binom{n-1}{2}.\]
Hence the assertion is true.

Assume now that the assertion is true if $|V(T)| \leq n-m$. Let $G = T \cup_a K_m$.   Since $E(T)\neq \emptyset$, it follows from \Cref{pendantexist} that there exists $u\in V(T)$ such that $\deg u\neq 1$ and $|S_u|\leq 1$. We now split the remaining proof into two cases.

\noindent
\textbf{Case 1:} $u\neq a$.\\
Let $e= \{u,v\}$ with $\deg_G v = 1$ and let $G' = G \setminus v$. Then $G' = (T\setminus v) \cup_a K_m$ and $J_{G'} = J_{G\setminus e}$. Note that $\deg_{G'} u = \deg_G u - 1$ and $\deg_{G'} w = \deg_G w$ for all $w \neq u$. From \Cref{maintheo}, we have $\beta_1(J_G) = \beta_1(J_{G\setminus e}) + \beta_{0}(J_{(G\setminus e)_e})$. We now compute the two terms on the right hand side of this equation. It follows by induction that 
\begin{eqnarray*}
        \beta_1(J_{G\setminus e}) &= &\binom{n-2}{2}+2\binom{m}{3}+\sum_{w\notin V(K_m), w\neq u}\binom{\deg_{G'} w}{3}+\binom{\deg_G u-1}{3}\\ & &+\binom{\deg_G a-m+1}{3}+ (n-m-2)\binom{m-1}{2} + (m-1)\binom{\deg_G a -m+1}{2}.
\end{eqnarray*}

Now, $(G\setminus e)_e$ is obtained by adding $\binom{\deg u-1}{2}$ edges to $E(G\setminus e)$. Since $T$ is a tree and $G=T \cup_a K_m$, we have $E(G) = n-m+\binom{m}{2}$. Hence, $G\setminus e$ has $n-m-1 + \binom{m}{2} = n-2+\binom{m-1}{2}$ edges. This means that:
    \[\beta_0(J_{(G\setminus e)_e}) =|E((G\setminus e)_e)| = n-2 + \binom{m-1}{2} +\binom{\deg_G u-1}{2}.\]
Therefore, 
    \begin{eqnarray*}
        \beta_1(J_{G}) &= & \beta_1(J_{G\setminus e}) + \beta_{0}(J_{(G\setminus e)_e}) \\
        & = & \binom{n-2}{2}+2\binom{m}{3}+\sum_{w\notin V(K_m), w\neq u}\binom{\deg_G w}{3}+\binom{\deg_G u-1}{3} \\
        & &+ \binom{\deg_G a-m+1}{3} + (n-m-2)\binom{m-1}{2} + (m-1)\binom{\deg_G a -m+1}{2}\\ & &+  n-2 + \binom{m-1}{2} +\binom{\deg_G u-1}{2}\\
        &= & \binom{n-1}{2}+2\binom{m}{3}+\sum_{w\notin V(K_m)}\binom{\deg_G w}{3}+\binom{\deg_G  a-m+1}{3}\\ & &+(n-m-1)\binom{m-1}{2} +(m-1)\binom{\deg_G a -m+1}{2}.
    \end{eqnarray*}
Therefore, we obtain our desired formula.

\noindent
\textbf{Case 2:} $u=a$.

\noindent
Let $e= \{a,v\}$ with $\deg v = 1$. Then, as before, we apply induction to get
    \begin{eqnarray*}
        \beta_1(J_{G\setminus e}) &= & \binom{n-2}{2}+2\binom{m}{3}+\sum_{w\notin V(K_m)}\binom{\deg_G w}{3}+ \binom{\deg_G a-m}{3}\\ & &+ (n-m-2)\binom{m-1}{2}+(m-1)\binom{\deg_G a -m}{2}.
    \end{eqnarray*}
There are $\binom{\deg_G a-m}{2}+(m-1)\binom{\deg_G a-m}{1}$ new edges in $(G\setminus e)_e$. Thus 
    \[\beta_0(J_{(G\setminus e)_e}) = |E(G\setminus e)_e| = n-2+\binom{m-1}{2}+\binom{\deg_G a-m}{2} + (m-1)\binom{\deg_G a-m}{1}.\]
Using \Cref{maintheo} and the identity $\binom{n}{r} = \binom{n-1}{r}+\binom{n-1}{r-1}$ appropriately, we get:
\begin{eqnarray*}
    \beta_1(J_{G}) & = & \binom{n-2}{2}+2\binom{m}{3}+\sum_{w\notin V(K_m)}\binom{\deg_G w}{3}+ \binom{\deg_G a-m}{3}\\
    & &+ (n-m-2)\binom{m-1}{2}+(m-1)\binom{\deg_G a -m}{2}\\
    & &+ n-2+\binom{m-1}{2}+\binom{\deg_G a-m}{2} + (m-1)\binom{\deg_G a-m}{1} \\
   & = & \binom{n-1}{2}+2\binom{m}{3}+\sum_{w\notin V(K_m)}\binom{\deg_G w}{3}+\binom{\deg_G  a-m+1}{3}\\
   & & +(n-m-1)\binom{m-1}{2}
    +(m-1)\binom{\deg_G a -m+1}{2}.
\end{eqnarray*}
Thus, we get the desired formula. This completes the proof.
\end{proof}

As an immediate consequence, we recover \cite[ Theorem 3.1]{jayanthan_almost_2021}:
\begin{corollary}
    Let $T$ be a tree on $[n]$.  Then
\[
        \beta_1(J_T) = 
        \binom{n-1}{2}+\sum_{w \in V(T)}\binom{\deg_T w}{3}.
\]        
\end{corollary}

\begin{proof}
If $G = T$, it can be trivially written as $G = T\cup_a K_1$, where $V(K_1) = \{a\}$. Therefore, taking $m=1$ in \Cref{T+K_m}
we get the desired formula.
\end{proof}

We now compute the second Betti number of a tree using
\Cref{T+K_m} and \Cref{maintheo}.
This Betti number also depends upon the
number of induced subgraphs
isomorphic to the following caterpillar tree.  We first fix the notation for this graph.

\begin{definition}
Let $P$ be the graph with $V(P)=[6]$ and 
$E(P) = \{\{1,2\}, \{2,3\},\\ \{3,4\}, \{2,5\}, \{3,6\} \}$. Given a tree $T$, we define $\mathcal{P}(T)$ to be the collection of all subgraphs of $T$ which are isomorphic to $P$, as shown in \Cref{fig:graph6}. Let $P(T) = |\mathcal{P}(T)|$.
\end{definition}

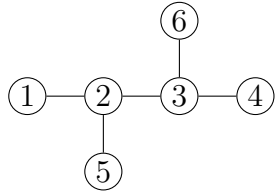
\begin{figure}[ht]
    \centering

\begin{tikzpicture}[every node/.style={circle, draw, fill=white!60, inner sep=1.5pt}, node distance=2cm]
    \node (1) at (0, 0) {1};
    \node (2) at (1, 0) {2};
    \node (3) at (2, 0) {3};
    \node (4) at (3, 0) {4};
    \node (5) at (1, -1) {5};
    \node (6) at (2, 1) {6};

    \draw (1) -- (2);
    \draw (2) -- (3);
    \draw (3) -- (4);
    \draw (2) -- (5);
    \draw (3) -- (6);
\end{tikzpicture}

    \caption{The graph $P$}
    \label{fig:graph6}
\end{figure}

\begin{example}\label{ex:pt}
Consider the graph $G$ of \Cref{fig:example of P} with $V(G) = [7]$ and $$E(G) = \{\{1,2\}, \{2,3\}, \{3,4\}, \{2,5\},\\ \{3,6\}, \{3,7\}\}.$$ For this graph, the collection $\mathcal{P}(G)$ will be the induced subgraphs on the following collections of vertices: $\mathcal{P}(G)=\{\{1,2,3,4,5,6\}, \{1,2,3,5,6,7\}, \{1,2,3,4,5,7\}\}$. Hence, $P(G)=3$.

\begin{figure}[ht]
    \centering
    \begin{tikzpicture}[every node/.style={circle, draw, fill=white!60, inner sep=1.5pt}, node distance=2cm]
    \node (1) at (0, 0) {1};
    \node (2) at (1, 0) {2};
    \node (3) at (2, 0) {3};
    \node (4) at (3, 0) {4};
    \node (5) at (1, -1) {5};
    \node (6) at (2, 1) {6};
    \node (7) at (2, -1) {7};

    \draw (1) -- (2);
    \draw (2) -- (3);
    \draw (3) -- (4);
    \draw (2) -- (5);
    \draw (3) -- (6);
    \draw (3) -- (7);
\end{tikzpicture}
    \caption{The graph $G$}
    \label{fig:example of P}
\end{figure}
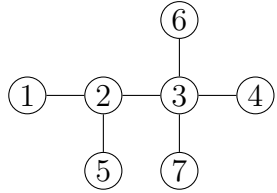

\end{example}

\begin{theorem}\label{betti2tree}
Let $T$ be a tree on $[n]$, and let $J_T$ be its binomial edge ideal. Then
    \[\beta_2(J_T) = \binom{n-1}{3}+ 2\sum_{w \in V(T)}\binom{\deg_T w}{4}+\sum_{w \in V(T)}\binom{\deg_T w}{3}(1+|E(T\setminus w)|)+P(T).\]
\end{theorem}

\begin{proof}
We prove the assertion by induction on $n$. If $n=2$, then $T$ is an edge. Since $J_T$ is a principal ideal, we have $\beta_{2}(J_T) = 0$, which agrees with the above formula. Now, assume that $n > 2$ and that the above formula is true for trees with $V(T)\leq n-1$. 
    
Let $T$ be a tree with $|V(T)|=n$. We know from \Cref{pendantexist} that there exists a vertex $u$ such that $\deg u>1$ and $|S_u|\leq 1$. Let $e = \{u,v\}$ be an edge such that $v$ is a pendant vertex. If $S_u = \emptyset$, then $T = K_{1,n-1}$. In this situation, the expression in the theorem statement reduces to $\binom{n-1}{3} + 2\binom{n-1}{4} + \binom{n-1}{3}.$  It is an easy verification that this number matches with the formula we obtained in \Cref{star}. 

We now assume that $|S_u| = 1$.
By the choice of $u$, we can see that $(T\setminus e)_e = (T\setminus v)\cup_a K_m \sqcup \{v\}$, where $S_u = \{a\}$ and $m = \deg_T u$. Let $G' = (T\setminus v)\cup_a K_m$. Then $|V(G')| = n-1$ and $J_{G'} = J_{(T\setminus e)_e}$.  
Observe that $\deg_{(T\setminus e)_e} a = \deg_T a + m-2$. Thus, from \Cref{T+K_m}, we get
    \begin{eqnarray*}
    \beta_1\left(J_{(T\setminus e)_e}\right) &= & \binom{n-2}{2} +2\binom{m}{3} + \sum_{w\notin V(K_m)}\binom{\deg_{(T\setminus e)_e} w}{3} +\binom{\deg_{(T\setminus e)_e} a-m+1}{3}\\ & &+(n-m-2)\binom{m-1}{2} + (m-1)\binom{\deg_{(T\setminus e)_e} a -m+1}{2}\\
    &= & \binom{n-2}{2} +2\binom{\deg_T u}{3} + \sum_{w\notin V(K_m)}\binom{\deg_T w}{3} +\binom{\deg_T  a-1}{3}\\
    & &+(n-\deg_T u-2)\binom{\deg_T u-1}{2} + (\deg_T u-1)\binom{\deg_T a-1}{2}.
    \end{eqnarray*}
Let $T' = T\setminus v$. Then $J_{T'} = J_{T\setminus e}$. Note that $|V(T')| = n-1,$ $\deg_{T'} u = \deg_T u-1$, and 
$\deg_{T'}x = \deg x$ for all $x \in V(T) \setminus\{u\}.$
Additionally $|E(T'\setminus u)| = |E(T \setminus u)|$ and $|E(T' \setminus w)| = |E(T \setminus w) | -1$ for all $w \neq u$. By the induction hypothesis, 
\begin{eqnarray*}
\beta_2(J_{T'})  & = & \binom{n-2}{3} + 2\sum_{w\neq u}\binom{\deg_T w}{4} + 2\binom{\deg_T u-1}{4} \\
        & &+\sum_{w\neq u}\binom{\deg_T w}{3}(|E(T\setminus w)|)+\binom{\deg_T u-1}{3}(|E(T \setminus u)|+1)+P(T').  
\end{eqnarray*}
Thus, it follows from \Cref{maintheo} that
    \begin{eqnarray*}
        \beta_2(J_{T}) &= & \binom{n-2}{3}+ 2\sum_{w\neq u}\binom{\deg_T w}{4}+ 2\binom{\deg_T u-1}{4} \\
        & &+\sum_{w\neq u}\binom{\deg_T w}{3}(|E(T\setminus w)|)+\binom{\deg_T u-1}{3}(|E(T \setminus u)|+1)+P(T')\\
        & &+\binom{n-2}{2}+2\binom{\deg_T u}{3}+\sum_{w\notin V(K_m)}\binom{\deg_T w}{3}+\binom{\deg_T a-1}{3}\\ & &+(n-\deg_T u-2)\binom{\deg_T u-1}{2}+(\deg_T u-1)\binom{\deg_T a-1}{2}.
    \end{eqnarray*}
Note that for all $w \in N_{T'}(u) \setminus \{a\}$, $\deg_{T'}(w) = 1$. Thus $\binom{\deg_{T'} w}{3} = 0$  for all $w\in N_{T'}(u) \setminus \{a\}$. Hence, none of the $w$, $w \neq a$, for which $\binom{\deg_T w}{3} \neq 0$ belong to $V(K_m)$ in $(T\setminus e)_e$. Thus we can write 
\[\sum_{w\neq u}\binom{\deg_T w}{3}(|E(T\setminus w)|) + \sum_{w\notin V(K_m)}\binom{\deg_T w}{3} = \sum_{w\neq u}\binom{\deg_T w}{3}(|E(T\setminus w)|+1).\]
To compare $P(T)$ and $P(T\setminus e)$, observe that the only elements of $\mathcal{P}(T)$ which are not in $\mathcal{P}(T\setminus e)$ are the induced subgraphs which contain the edge $e$. Since $a$ is the only neighbor of $u$ having degree more than one, 
the total number of such graphs is $(\deg_T u -2)\binom{\deg_T a-1}{2}$. Thus $P(T\setminus e) = P(T) - (\deg_T u -2)\binom{\deg_T a-1}{2}.$ Note also that $|E(T\setminus u)| =n-\deg_T u -1$.

Incorporating the above observations in the expression for $\beta_2(J_T)$, and using the identity $\binom{n}{r} = \binom{n-1}{r-1} + \binom{n-1}{r}$, we get
\footnotesize 
\begin{eqnarray*}
     \beta_2(J_T) &= & \binom{n-1}{3} + 2\sum_{w\neq u}\binom{\deg_T w}{4} + 2\binom{\deg_T u-1}{4}+\sum_{w\neq u,a}\binom{\deg_T w}{3}(|E(T\setminus w)|+1) \\
    & &+\binom{\deg_T a}{3}(|E(T\setminus a)|)+\binom{\deg_T u-1}{3}(|E(T\setminus u)|+1)+P(T)+\binom{\deg_T a-1}{2}\\ 
    & &+2\binom{\deg_T u}{3}+\binom{\deg_T a-1}{3}+(|E(T\setminus u)|-1)\binom{\deg_T u-1}{2}\\
    &= & \binom{n-1}{3}+ 2\sum_{w\neq u}\binom{\deg_T w}{4} + 2\binom{\deg_T u-1}{4} +\sum_{w\neq u,a}\binom{\deg_T w}{3}(|E(T\setminus w)|+1)\\
    & &+\binom{\deg_T a}{3}(|E(T\setminus a)|+1)+\binom{\deg_T u}{3}(|E(T\setminus u)|+1)\\ & &+P(T)+2\binom{\deg_T u}{3}-2\binom{\deg_T u-1}{2}\\
    &= & \binom{n-1}{3}+ 2\sum_{w\neq u}\binom{\deg_T w}{4} + 2\binom{\deg_T u-1}{4}+\sum_{w}\binom{\deg_T w}{3}(|E(T\setminus w)|+1)\\
    & &+P(T) +2\binom{\deg_T u-1}{3} \\
    &= & \binom{n-1}{3} + 2\sum_{w}\binom{\deg_T w}{4} +\sum_{w}\binom{\deg_T w}{3}(1+|E(T\setminus w)|)+P(T).
    \end{eqnarray*}
    \normalsize
    We have now completed the proof.
\end{proof}

It can be seen that \Cref{betti2tree} builds on  \cite[Theorem 3.1]{jayanthan_almost_2021}. We conclude our article by computing certain graded Betti numbers of binomial edge ideals of trees.

\begin{theorem}\label{thirdrow}
Let $T$ be a tree and $J_T$ be its corresponding binomial edge ideal. Then,
\[\beta_{k,k+3}(J_T) = \sum_{w\in V(T)}k\binom{\deg_T w+1}{k+2}\text{ for all k $\geq 2$}.\]
\end{theorem}

\begin{proof}
We prove the assertion by induction on $|V(T)|$. Let $|V(T)|=n=2$. Then $J_T$ is the binomial edge ideal of a single edge. Since this is a principal ideal generated in degree $2$, $\beta_{k,k+3}(J_T)=0$ for all $k\geq 2$, which agrees with the formula. Suppose the assertion is true for all trees with $n-1$ vertices. Let $T$ be a tree with $|V(T)| = n$. Using \Cref{pendantexist}, consider $e=\{u,v\} \in E(T)$, where $u$ is such that $\deg u>1$ and $|S_u|\leq 1$. Then, using \Cref{maintheo}, we get
    \[\beta_{k,k+3}(J_T) = \beta_{k,k+3}(J_{T\setminus e})+ \beta_{k-1,k+1}(J_{(T\setminus e)_e}).\]
Let $T' = T \setminus v$. Then $J_{T'} = J_{T\setminus e}$, $\deg_{T'} u = \deg_T u - 1$ and $\deg_{T'} w = \deg_T w$ for all $w \in V(T') \setminus u$. Also, $(T\setminus e)_e$ is a clique sum of a tree and a complete graph, with the size of the complete graph equal to $\deg u$. Hence using the inductive hypothesis and \Cref{linearbinom} we get:
    \begin{align*}
        & \beta_{k,k+3}(J_{T\setminus e}) = \sum_{w\neq u}k\binom{\deg_T w+1}{k+2} + k\binom{\deg_T u}{k+2},~~\mbox{and}\\
        & \beta_{k-1,k+1}(J_{(T\setminus e)_e}) = k\binom{\deg_T u}{k+1}.
    \end{align*}
Substituting these values into \Cref{maintheo} we get:
    \[\beta_{k,k+3}(J_T) = \sum_{w\neq u}k\binom{\deg_T w+1}{k+2} + k\binom{\deg_T u}{k+2}+k\binom{\deg_T u}{k+1} = \sum_{w}k\binom{\deg_T w+1}{k+2}.\]
\end{proof}

\begin{example}
To illustrate some of the results of this section, consider
the tree $T$ described in \Cref{ex:pt}.  The graded Betti table
of $J_T$ is given below:
\begin{verbatim}
                0  1  2  3  4 5
        total:  6 20 41 43 21 4
            2:  6  .  .  .  . .       
            3:  . 20 12  3  . .      
            4:  .  . 29 40 21 4           
\end{verbatim}
It follows from \Cref{betti2tree} that for this tree 
$$\beta_2(J_T) = \binom{6}{3} + 2 \binom{4}{4} + \binom{4}{3}(1+2) + \binom{3}{3}(1+3) + 3 = 41.$$ 
Additionally, the Betti numbers of
the form $\beta_{k,k+3}(J_T)$ with $k \geq 2$ satisfy \Cref{thirdrow} 
since $\beta_{2,5}(J_T) = 2\left[\binom{4}{4} + \binom{5}{4}\right] = 12$ and $\beta_{3,6}(J_T) = 3 \binom{5}{5} = 3$.
\end{example}
\noindent
{\bf Acknowledgments.}
Some of the results in this paper first appeared in the MSc thesis 
of Sivakumar.  
Van Tuyl’s research is supported by NSERC Discovery Grant 2024-05299.

\bibliographystyle{amsplain}
\bibliography{Bibliography}

\end{document}